%% file: exoticR4-1.tex
\documentclass[11pt, reqno]{amsart}
\usepackage[dvipsnames,usenames]{color}
\usepackage[colorlinks=true, urlcolor=NavyBlue, linkcolor=NavyBlue, citecolor=NavyBlue]{hyperref}
\usepackage{cleveref}
\usepackage{graphicx}
\usepackage{epsfig}
\usepackage[latin1]{inputenc}
\usepackage{amsmath}
\usepackage{amsfonts}
\usepackage{amssymb}
\usepackage{amsthm}
\usepackage{amscd}
\usepackage{verbatim}
\usepackage{subfigure}
\usepackage{caption}
\usepackage{pinlabel}
\usepackage{stmaryrd}
\usepackage{enumerate, enumitem}
\usepackage{todonotes}
\usepackage{bm}
\usepackage{thmtools}
\usepackage{thm-restate}
\usepackage{lipsum}
\usepackage{setspace}
\usepackage{mathtools}
\usepackage[all]{xypic}
\usepackage[abs]{overpic}
\usepackage{color}
\usepackage[normalem]{ulem}
\usepackage[alphabetic,backrefs,msc-links]{amsrefs}
\usepackage{mathdots}
\usepackage{comment}

\allowdisplaybreaks

\usepackage{tikz}
\usetikzlibrary{arrows}
\usetikzlibrary{decorations.pathreplacing}
\usepackage{verbatim}
\usetikzlibrary{cd}
\tikzset{taar/.style={double, double equal sign distance, -implies}}
\tikzset{amar/.style={->, dotted}}
\tikzset{dmar/.style={->, dashed}}
\tikzset{aar/.style={->, very thick}}

\newtheorem{theorem}{Theorem}[section]

\newtheorem{lemma}[theorem]{Lemma}
\newtheorem{proposition}[theorem]{Proposition}

\newtheorem{corollary}[theorem]{Corollary}

\theoremstyle{definition}
\newtheorem{definition}[theorem]{Definition}

\theoremstyle{remark}
\newtheorem{remark}[theorem]{Remark}

\def\F{\mathbb{F}}

\def\Q{\mathbb{Q}}
\def\R{\mathbb{R}}
\def\Z{\mathbb{Z}}

\def\cF{\mathcal{F}}

\def\cT{\mathcal{T}}
\def\cR{\mathcal R}

\def \gr {\operatorname{gr}}

\def\Cone{\operatorname{Cone}}

\def\d{\partial}

\def\spinc{\textrm{Spin}^c}

\def\red{\textup{red}}

\def\Cone{\operatorname{Cone}}

\def\HF {\mathit{HF}}
\newcommand\HFhat{\widehat{\HF}}

\def\CFK{\mathit{CFK}}
\def\CFKm{\mathit{CFK}^-}
\def\CFKi{\mathit{CFK}^\infty}

\newcommand\HFKhat{\widehat{\mathit{HFK}}}

\def\spinc {{\operatorname{spin^c}}}

\def\s{\mathfrak s}

\def\HFKhat{\widehat{\mathit{HFK}}}
\def\CFKhat{\widehat{\mathit{CFK}}}

\def\HFred{\HF_{\operatorname{red}}}

\newcommand{\lab}[1]{$\scriptstyle #1$}

\newcommand\Wh{\mathit{Wh}}
\newcommand\HE{\mathit{HE}}
\newcommand\sh{\mathrm{sh}}
\newcommand\bbR{\mathbb{R}}
\newcommand\caR{\mathcal{R}}
\newcommand\CH{\mathit{CH}}
\newcommand\std{\mathrm{std}}

\newcommand{\tyeq}[1]{{\color{blue}TL: #1}}

\author[S. Eli]{Sean Eli}
\thanks{SE was partially supported by NSF grant DMS-2203312.}
\address {School of Mathematics, Georgia Institute of Technology, Atlanta, GA 30332}
\email{seaneli@gatech.edu}

\author[J.\ Hom]{Jennifer Hom}
\thanks{JH was partially supported by NSF grants DMS-2104144 and DMS-2506400, and Georgia Tech's Elaine M. Hubbard Faculty Fellowship.}
\address {School of Mathematics, Georgia Institute of Technology, Atlanta, GA 30332}
\email{hom@math.gatech.edu}

\author[T. Lidman]{Tye Lidman}
\thanks{TL was partially supported by NSF grants DMS-2105469 and DMS-2506277 and a Simons Travel Support award.}
\address{Department of Mathematics, North Carolina State University, Raleigh, NC 27607}
\email{tlid@math.ncsu.edu}

\numberwithin{equation}{section}

\title{Distinguishing exotic $\mathbb{R}^4$'s with Heegaard Floer homology}

\begin{document}

\begin{abstract}
Attaching a Casson handle to a slice disk complement yields a smooth 4-manifold that is homeomorphic to $\R^4$. 
We show that if two slice knots have sufficiently different knot Floer homology, then the resulting exotic $\R^4$'s made using the simplest positive Casson handle are not diffeomorphic, giving us a countably infinite family of pairwise nondiffeomorphic chiral exotic $\R^4$'s.  Our main tool is Gadgil's end Floer homology and we use this to produce families of exotic $\R^4$ with various phenomena.  
As an application, we reprove a result of Bi{\v z}aca-Etnyre that $Y \times \R$, where $Y$ is any closed $3$-manifold, has infinitely many distinct smooth structures. 
\end{abstract}

\maketitle

\input{intro}
\input{noncompactbackground}

\input{endFloer}

\input{HFprelim}
\input{cobordismmap}
\input{Whitehead}

\input{proofs}

\bibliographystyle{alpha}
\bibliography{bib}

\end{document}

%% file: intro.tex
\section{Introduction}

A surprising phenomenon in smooth manifold topology unique to four dimensions is the existence of exotic $\mathbb{R}^4$'s: smooth manifolds homeomorphic but not diffeomorphic to $\mathbb{R}^4$. These manifolds were first confirmed to exist in the 1980's as a consequence of the smooth failure (Donaldson \cite{donaldson}) and topological success (Freedman \cite{freedman}) of the $h$-cobordism theorem for 4-manifolds. Shortly after their discovery there was a wave of progress \cite{gompfinfinite,taubesend, freedmandemichelis, gompfmenagerie , bizacagompf} towards the classification problem for exotic $\mathbb{R}^4$'s,
but many elementary questions in this realm still remain open \cite{bizacagompf, gompfstipsicz, Taylor}. 

We consider \textit{slice $\R^4$'s}, potentially exotic $\R^4$'s obtained by attaching a Casson handle to a slice disk complement and removing the remaining boundary. See Figure \ref{fig:example} for an example. 
A natural question to ask is when this construction produces distinct exotic $\R^4$'s. In the case that the Casson handle is fixed and the slice disk complement is varied, only one diffeomorphism type of exotic $\R^4$ has been known to arise.
We distinguish an infinite family using the knot Floer homology of $K$.

\begin{theorem}\label{thm:exoticR4}
Let $\caR$ be built from attaching the Casson handle $\CH^+$ to a slice disk complement $(B^4,S^3)-\nu(D^2,K)$ for a nontrivial slice knot $K$ and removing the boundary. 
\begin{enumerate}
	\item\label{it:exoticR4-1} The manifold $\caR$ is an exotic $\R^4$. 
	\item\label{it:exoticR4-2} Moreover, if $K_1$ and $K_2$ are two nontrivial slice knots whose knot Floer homology $\HFKhat_\red$ has different maximal nontrivial Maslov gradings, then the exotic $\R^4$'s $\caR_1$ and $\caR_2$ built as above are not diffeomorphic with any choice of orientations. 
\end{enumerate}
\end{theorem}
\noindent These exotic $\R^4$'s are chiral. See Section \ref{sec:Whitehead} for the definition of $\HFKhat_\red$.  (For nontrivial, thin, slice knots, the maximal nontrivial grading of $\HFKhat_\red$ is the same as that of $\HFKhat$.)   
Note that the above result does not depend on the choice of slice disk. The proof relies on Gadgil's end Floer homology \cite{gadgil}.

Theorem \ref{thm:exoticR4} immediately gives infinitely many slice $\R^4$'s made with the same Casson handle:

\begin{corollary}\label{cor:exoticR4}
Let $K_n$ denote the ribbon knots $T_{2,n} \# T_{2,-n}$, for $n \geq 3$ and odd.  Then the exotic $\R^4$'s $\caR_n$ made with standard disk complements as in Theorem \ref{thm:exoticR4} are pairwise nondiffeomorphic.
\end{corollary}

The exotic $\R^4$'s of Theorem \ref{thm:exoticR4} are made using the Casson handle $\CH^+$ with a single positive double point at each stage. By considering orientations it follows that any slice $\R^4$ made with the simplest all-negative Casson handle $CH^-$ is exotic, and we can distinguish members of this family by Theorem \ref{thm:exoticR4} \eqref{it:exoticR4-2}. Furthermore this shows that any slice $\R^4$ made using a linear-chain Casson handle with only finitely many double points of one sign is exotic: see Remark \ref{rem:mixedsigns}. We can understand slice $\R^4$'s made with more general Casson handles as well, and obtain exotica with different properties than those of Theorem \ref{thm:exoticR4}:

\begin{theorem}\label{thm:branching}
Let $K$ be the Whitehead double of a nontrivial slice knot, and let $\caR$ be built from attaching a Casson handle $\CH$ to a slice disk complement $(B^4,S^3)-\nu(D^2,K)$ and removing the boundary. If $\CH$ has an infinite positive chain, we have:
\begin{enumerate}
\item \label{it:branching1} $\caR$ is an exotic $\R^4$.
\item \label{it:branching2} If, in addition, $\CH$ has an infinite negative chain, then $\caR$ is not diffeomorphic to any exotic $\R^4$ of Theorem \ref{thm:exoticR4} with any choice of orientations. 
\end{enumerate}
\end{theorem}

An example of a Casson handle with infinite negative and positive chains is $\CH^*$, shown truncated in Figure \ref{fig:branchinghandle}. The class of Casson handles having an infinite positive chain includes the Stein Casson handles described by Gompf \cite{gompfstipsicz}, which are those used in Gadgil's slice $\R^4$'s \cite{gadgil}. In particular item (\ref{it:branching1}) of Theorem \ref{thm:branching} gives another proof that the slice $\R^4$'s of Gadgil are exotic, in the case where the starting knot is a Whitehead double. 

The $\R^4$'s given in item (\ref{it:branching2}) of Theorem \ref{thm:branching} have nonvanishing end Floer homology with any choice of orientation, unlike the $\R^4$'s of Theorem \ref{thm:exoticR4}. Complementing this, we have another type of exotic $\R^4$ not diffeomorphic to those of Theorems \ref{thm:exoticR4} or \ref{thm:branching} with any choice of orientations:

\begin{theorem}\label{thm:endfloerzero}
Let $\caR$ be built from attaching the Casson handle $\CH^+$ to a slice disk complement $(B^4,S^3)-\nu(D^2,K)$ for a nontrivial slice knot $K$ and removing the boundary. The manifold $\caR\natural\overline{\caR}$ is an exotic $\R^4$ whose end Floer homology vanishes with any choice of orientations.
\end{theorem}



The exotic $\bbR^4$'s $\caR_n$ of Corollary \ref{cor:exoticR4} can be end-summed on to certain noncompact 4-manifolds to obtain infinitely many different smoothings, detectable via end Floer homology. Rather than giving the most general statement, we focus on a particular instance of this which partially recovers the following result due to Bi{\v z}aca and Etnyre \cite{bizacaetnyre}, though our smoothings have different properties (see Remark \ref{rem:taylorinvariant}).

\begin{theorem}\label{thm:endsumproduct}
Let $Y$ be any closed 3-manifold, not necessarily orientable. Then $Y\times \bbR$ has infinitely many nondiffeomorphic smooth structures.
\end{theorem}

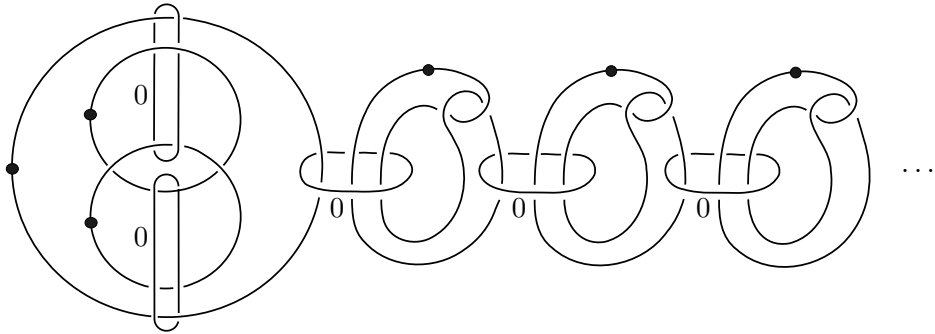
\begin{figure}[ht]
\centering
\labellist
	\pinlabel {$0$} at 81 55
	\pinlabel {$0$} at 81 115
	\pinlabel {$0$} at 164 67
	\pinlabel {$0$} at 241 67
	\pinlabel {$0$} at 319 67
	\pinlabel {$\hdots$} at 410 83
\endlabellist
\includegraphics[scale=0.9]{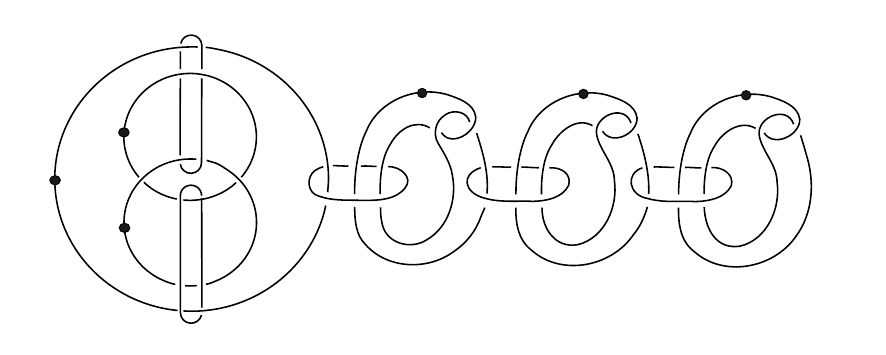}
\caption{A Kirby diagram for an exotic $\R^4$ obtained by attaching the Casson handle $\CH^+$ to a ribbon disk complement for $T_{2,3}\# -T_{2,3}$. }
\label{fig:example}
\end{figure}


\subsection{Background and context.}

Recall that \textit{ribbon $\mathbb{R}^4$'s} \cite{freedmandemichelis, bizacagompf} are a subfamily of exotic $\mathbb{R}^4$'s which embed in $\mathbb{R}^4_{\std}$ and admit explicit Kirby diagrams. These exotic $\mathbb{R}^4$'s were initially found inside non-product $h$-cobordisms between closed 4-manifolds, enveloping the non-product portion. Ribbon $\mathbb{R}^4$'s admit a simple description: choose a ribbon disk complement in $B^4$, attach a 0-framed Casson handle along the meridian of the deleted disk, and then delete the remaining boundary.  (For more detail on Casson handles, see Section~\ref{sec:background}.)   Any manifold constructed this way is homeomorphic to $\mathbb{R}^4$ by Freedman's theorem that Casson handles are homeomorphic to open 2-handles. This simple description has enabled the study of ribbon $\bbR^4$'s from various perspectives including diffeomorphism groups \cite{gompfmenagerie}, Riemannian geometry \cite{curvatureR4}, and Floer homology \cite{gadgil}. 
However, the central problem of understanding exactly how the choices of ribbon knot, disc, and Casson handle influence the diffeomorphism type of the ribbon $\mathbb{R}^4$ has seen limited progress since its introduction in the 90's \cite{bizacagompf, freedmandemichelis}. Similarly one can define \textit{slice $\R^4$'s} and ask the same questions; the current article is concerned with this potentially larger class. 

We summarize previous results below. Call a smooth manifold homeomorphic to $\mathbb{R}^4$ \textit{exotic} if it is not diffeomorphic to the standard $\mathbb{R}^4$. 


\begin{enumerate}
    \item 
    If we combine the unknot and its standard disc complement with any Casson handle, the result is the standard $\mathbb{R}^4$. This follows by an engulfing argument outlined in \cite{freedman}. Gompf pointed out that if we use the unknot and any slice disc, and any Casson handle, the result is standard. This is because of Casson's proof that Whitehead doubles of slice discs for the unknot are standard. 
    \item 
    Suppose $(W^5;X_1,X_2)$ is an $h$-cobordism between an exotic pair of closed 4-manifolds. De Michelis and Freedman \cite{freedmandemichelis} show that there exists a ribbon $\bbR^4$ $\cR$ and a topologically trivial proper $h$-cobordism $(U; \cR_0, \cR_1) \subset (W^5;X_1,X_2)$ where $W - U$ is a smooth product and $\cR_0\cong \cR_1\cong \cR$. In general this $\cR$ is made from a ribbon link complement. 
    In \cite{freedmandemichelis} it is shown that by varying the Casson handles and keeping the disc complement fixed, uncountably many distinct ribbon $\mathbb{R}^4$'s are obtained. This is shown by a cardinality argument, and in particular the explicit Casson handles giving different diffeomorphism types are not known. 
    \item 
    The ``simplest ribbon $\mathbb{R}^4$" of Bi{\v z}a{\v c}a-Gompf \cite{bizacagompf}, obtained using a disc for the $(-3,-3,3)$-pretzel knot and the simplest Casson handle $\CH^+$, is exotic: while it is not known to explicitly appear in an $h$-cobordism in the required way, it is similar enough to another exotic $\mathbb{R}^4$ with this property to allow a similar proof of exoticness.
    \item\label{it:BG-mirror} 
    The simplest ribbon $\mathbb{R}^4$ of Bi{\v z}a{\v c}a-Gompf is shown in \cite{bizacagompf} to not admit an orientation preserving diffeomorphism to its mirror image, which is a ribbon $\mathbb{R}^4$ using a standard disc for the $(3,3,-3)$-pretzel knot and the simplest Casson handle with all negative clasps. Prior to the present work, this was the only known result  distinguishing ribbon $\mathbb{R}^4$'s made from different ribbon disk complements. 
        \item  
        For any slice disk complement, there exists some Casson handle such that the resulting slice $\mathbb{R}^4$ is exotic. This result is due to Gadgil \cite{gadgil} using end Floer homology, a direct limit of Heegaard Floer homology; this also follows from Item (\ref{it:exoticR4-1}) of our Theorem \ref{thm:exoticR4}. It follows from Gadgil's proof that any Casson handle chosen from a suitable subfamily of Stein Casson handles \cite{gompfstipsicz} yields an exotic $\R^4$.  This result does not distinguish between any of these exotic $\R^4$'s.
\end{enumerate}

We remark that the pair in \eqref{it:BG-mirror} above is the only previously known example where slice $\mathbb{R}^4$'s made from nontrivial, distinct slice knots are distinguished; note they are mirror images hence are orientation-reversing diffeomorphic. 
 
Regarding these previous results, early proofs of exoticness were indirect and only applicable to certain ribbon $\mathbb{R}^4$'s that 
arise in special situations;
Gadgil's end Floer homology \cite{gadgil} is a relatively recent breakthrough, providing the first invariant able to detect when a general ribbon $\mathbb{R}^4$ is exotic. The key idea is that any exotic $\mathbb{R}^4$ admits a decomposition into a compact 4-manifold combined with an infinite stack of cobordisms. Loosely, the end Floer homology is the direct limit of Heegaard Floer modules under the cobordism maps coming from this infinite stack. Using the end Floer invariant $\HE(X)$ Gadgil showed that for any slice disk complement, there exists some Casson handle such that the resulting ribbon $\mathbb{R}^4$ has nonvanishing end Floer homology, and thus is not diffeomorphic to $\mathbb{R}^4$. 
Moreover, Gadgil proved this result without fully describing the group $\HE(X)$: it was sufficient to show this group is nontrivial. 

In this paper we refine Gadgil's notion of admissible cobordisms in order to endow $\HE(X)$ with a well-defined absolute $\mathbb{Q}$-grading. Using this and a careful analysis of 2-handle cobordism maps, we are able to use end Floer homology to prove \ref{thm:exoticR4} and Corollary \ref{cor:exoticR4}.

\begin{remark}
Using similar techniques as in the proof of Theorem \ref{thm:endsumproduct} we can distinguish between various end-sums of the $\caR_n$.  See Remark~\ref{rmk:endsumRn} below.  
\end{remark}

\begin{remark}\label{rem:taylorinvariant}
Any exotic smoothing of $Y\times \R$ constructed in Theorem \ref{thm:endsumproduct} will have the same value of the Taylor invariant \cite{Taylor} as the standard smoothing of $Y\times \R$. Recall the Taylor invariant of a noncompact 4-manifold $X$ is the smallest $n \ge 0$ such that $X$ embeds smoothly into some closed, spin 4-manifold with $b_2^+ = b_2^- = n$. Previously, Bi{\v z}aca and Etnyre \cite{bizacaetnyre} proved there are infinitely many smoothings of $Y\times \bbR$ for any compact 3-manifold $M$, but all of their smoothings have different, arbitrarly large values of the Taylor invariant. Fang and Gompf \cite{fang1, gompfstipsicz, gompfgenera} find infinitely many smoothings for certain choices of $Y$ all with the same Taylor invariant, but the case for general closed $Y$ was not previously known.
\end{remark}

\begin{remark}
We do not give the precise statement here, but roughly, if one has a 4-manifold with one end and the end Floer homology is non-trivial and supported in relative gradings bounded from above, then end-summing with the various $\caR_n$ will produce an infinite family of exotica.   
\end{remark}

\subsection{Overview of proof strategy} We conclude the introduction with a brief overview of the proof strategy. Gadgil's end Floer homology is defined as a direct limit of a directed system built from cobordisms. In the situation at hand, each of these cobordisms consists of a 1-handle and a 2-handle. The Heegaard Floer cobordism map associated to a 1-handle attachment is standard, and so it remains to understand the 2-handle cobordism. As the attaching circle for the 2-handle is nontrivial in homology, it is difficult to compute the cobordism map directly. Thus, we appeal to the surgery exact triangle, shown in Figure \ref{fig:triangle}, to determine the cobordism map. 

Two of the three 3-manifolds in the surgery exact triangle are straightforward: $S^3_0(K) \# (S^1 \times S^2)$ and $S^3_0(\Wh(K))$, where $K$ itself is in general some iterated Whitehead double of the initial slice knot. The third manifold in the exact triangle is shown in Figure \ref{fig:KJinY}. We compute the Heegaard Floer homology of this manifold using a combination of the surgery mapping cone formula of \cite{OS-integer}, the K\"unneth formula, and properties of $d$-invariants and L-space knots.

Once we know the Heegaard Floer homology of the three 3-manifolds in the surgery exact triangle, a grading and rank argument determines the cobordism map in question. The cobordism map calculation allows us to show that the end Floer homology is infinite dimensional in a particular grading; this grading depends on the knot Floer homology of the initial slice knot.

\subsection{Organization.}

In Section 2 we give relevant background on Casson handles, describe helpful exhaustions of noncompact 4-manifolds, and discuss the end-sum operation. In Section 3 we define the new notion of grading admissible exhaustions which we then use to give Gadgil's end Floer invariant an exhaustion-independent absolute grading. Section 4 reviews Heegaard Floer preliminaries to establish notation. In Section 5 we compute $\HF^+(S^3_0(\Wh(K))$ where $K$ is a slice knot and analyze the relevant cobordism map from $\HF^+(S^3_0(K)) \to \HF^+(S^3_0(\Wh(K)))$ using the surgery exact triangle.  Section 6 studies the knot Floer complex of Whitehead doubles. Finally, in Section 7 we prove Theorem \ref{thm:exoticR4}, Corollary \ref{cor:exoticR4}, Theorem \ref{thm:branching}, Theorem \ref{thm:endfloerzero}, and Theorem \ref{thm:endsumproduct}, and give remarks about generalizations.
	
Throughout, when we say the Whitehead double of a knot, we mean the positively clasped, untwisted Whitehead double. We work with Heegaard Floer homology with $\F=\Z/2\Z$ coefficients.

\subsection{Acknowledgements}
We are grateful for the 2025 Georgia International Topology Conference, where part of this work was done.  We also thank John Etnyre, Bob Gompf, Adam Levine, and JungHwan Park for helpful conversations.  TL thanks Georgia Tech for its hospitality during a visit in Summer 2025.  

%% file: noncompactbackground.tex
\section{Noncompact 4-manifold background}\label{sec:background}

\subsection{Casson handles} We give a brief description of Casson handles and their handle decompositions. For more background, see any of \cite{cassonthree, BennettR4, gompfstipsicz, freedman}. First we recall some facts about self-plumbed 2-handles, which are the building blocks of Casson handles. Suppose $(k,\partial_-k)$ is a 2-handle which has been self-plumbed finitely many times away from the attaching region. Here $\partial_-k$ denotes the image of the 2-handle's attaching region under the plumbing quotient map; we also write $\partial_+ k := \overline{\partial k \setminus \partial_- k}$ for the upper boundary of $k$. Let $D$ denote the image of the 2-handle's core disk $D^2 \times 0 \subset D^2\times D^2$ under the plumbing operations, so $D \subset k$ is an immersed disk with finitely many double points; we call $C := \partial D \subset \partial_- k$ the attaching circle of $k$. The plumbed handle $(k,\partial_- k)$ is determined up to diffeomorphism by the numbers of positive and negative double points of $D$. The attaching circle $C$ is framed so that attaching $k$ along a framed knot has the same effect on the intersection form as attaching an ordinary 2-handle along the same framed knot. This amounts to twisting the product framing on $\partial D^2\times 0$ coming from $\partial D^2\times D^2$ by twice the signed number of self-plumbings. Note each double point of the immersed core $D$ has a double point loop on $D$. We may isotope the double point loops to reside in $\partial_+k$, disjointly from each other. These loops are framed so that attaching ordinary 2-handles to $k$ along them yields $B^4$, where $C$ maps to a 0-framed unknot in $S^3$. 

A \textit{1-stage Casson Tower} $(T_1, \partial_-T_1)$ is a self-plumbed 2-handle $(k,\partial_-k)$. An \textit{n-stage Casson Tower} $(T_n, \partial_-T_n)$ is obtained inductively by attaching self-plumbed 2-handles to all of the double point loops of a height $n-1$ Casson tower, using the framings we described. A \textit{Casson handle} is an infinite Casson tower minus its upper boundary, described as follows. Start with a 1-stage Casson tower $(T_1,\partial_-T_1)$ and attach self-plumbed 2-handles to obtain a 2-stage tower $T_2$, a 3-stage tower $T_3$, and so on until we have a nested sequence of towers $T_1\subset T_2\subset T_3\subset \hdots$. A \textit{Casson handle} $(\CH, \partial_-\CH)$ is the infinite union of any sequence of $T_i$ constructed this way, minus all boundary except the interior of $\partial_-T_1$, which we denote $\partial_-\CH$. To attach a Casson handle along a framed knot, use the first stage attaching circle $C\subset \partial_- \CH$ with the framing coming from $T_1$. 

Freedman's theorem \cite{freedman} shows that any Casson handle $(\CH, \partial_-\CH)$ is homeomorphic rel boundary to an open 2-handle ($D^2\times \bbR^2$, $S^1\times \bbR^2$). By Donaldson's theorem \cite{donaldson} and the failure of the $h$-cobordism theorem in dimension 5, it follows that not all Casson handles are diffeomorphic rel boundary to the open 2-handle. It is not known whether this is true for all Casson handles.

We remark that any finite number of self-plumbings with choice of signs is possible at any stage in a general Casson handle; any Casson handles is determined up to diffeomorphism by a singly-rooted tree with signed, finite-degree vertices, and no finite branches. We say a Casson handle contains an \textit{infinite positive (resp. negative) chain} if its describing tree contains an infinite linear chain, starting from the root vertex, with all positive (resp. negative) signs. Figure \ref{fig:cassonhandle} is a Kirby diagram for a 3-stage Casson tower attached to $B^4$ along a 0-framed unknot. Each stage has exactly one positive double point. In the rest of this article, let $\CH^+$ denote the Casson handle with a single positive double point in each stage, and $T_i^+$ denote its $i$-stage approximations. Figure \ref{fig:branchinghandle} shows the first three stages of the simplest Casson handle with exactly one infinite positive chain and one infinite negative chain, attached to $B^4$ along a 0-framed unknot. In the rest of this article, this Casson handle will be denoted $\CH^*$.


\begin{figure}[ht]
\centering
\labellist
	\pinlabel {$0$} at 53 65
	\pinlabel {$0$} at 167 65
	\pinlabel {$0$} at 267 65

\endlabellist
\includegraphics[scale=0.9]{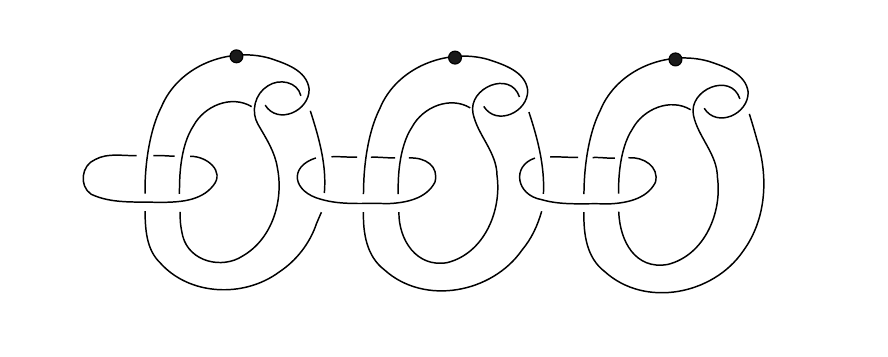}
\caption{A Kirby diagram for $B^4\cup T^+_3$. We see three copies of $T_1^+$: the 0-framed unknots are their attaching circles, and meridians of the dotted circles are their double point loops. Continuing the pattern infinitely to the right and deleting the boundary yields $\text{int}(B^4\cup \CH^+)$. }
\label{fig:cassonhandle}
\end{figure}

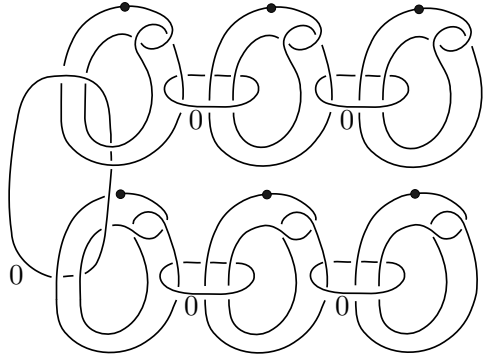
\begin{figure}[ht]
\centering
\labellist
	\pinlabel {$0$} at 104 46
	\pinlabel {$0$} at 178 33
	\pinlabel {$0$} at 242 33
	\pinlabel {$0$} at 180 112
	\pinlabel {$0$} at 244 112

\endlabellist
\includegraphics[scale=0.9]{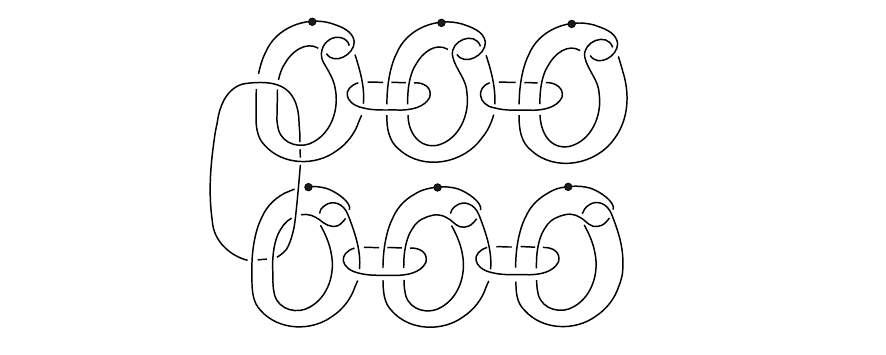}
\caption{A Kirby diagram for $B^4$ together with the first 3 stages of $CH^*$, attached along a 0-framed unknot in $S^3$.}
\label{fig:branchinghandle}
\end{figure}


\subsection{Exhaustions of slice $\mathbb{R}^4$'s}

Recall that slice $\mathbb{R}^4$'s are constructed by taking a slice disk complement in $B^4$, attaching a 0-framed Casson handle along the meridian of the deleted disk, and deleting the remaining boundary. Let $\mathcal{R}$ be a slice $\mathbb{R}^4$ constructed from a slice disk complement for $(D^2,K)\hookrightarrow (B^4, S^3)$ and the Casson handle $\CH^+$. We describe a useful compact exhaustion of $\mathcal{R}$. Let $X_0$ be the slice disk complement $B^4\setminus \nu(D)$. Let $X_1$ denote $X_0$ union a collar of $\partial X_0$, union a copy of $T_1^+$ attached along a 0-framed meridian of $D$. Inductively, define $X_i$ to be $X_{i-1}$ union a collar of $\partial X_{i-1}$, union a copy of $T_1^+$ attached along the previous tower's double point loop. By including the collars, we ensure that $\overline{X_i}\subset \mathrm{int} (X_{i+1})$ for each $i$. It is clear that $X_i$ is a compact exhaustion of $\mathcal{R}$. We state some properties of this exhaustion that foreshadow the application of Heegaard Floer homology.

\begin{lemma}\label{lem:sliceR4exhaustion}
Let the $X_i$ be as above. Then for all $i$:

\begin{enumerate}
\item \label{it:sliceR4exhaustion1} $X_i  \cong B^4\setminus \nu(\Wh^i(D))$ and $\partial X_i \cong S^3_0(\Wh^i(K))$.
\item \label{it:sliceR4exhaustion2} $X_{i+1}-  \mathrm{int}\, X_i$ is a cobordism from $S^3_0(\Wh^i(K))$ to $S^3_0(\Wh^{i+1}(K)))$ consisting of a 1- and 2-handle attached 
 as in Figure~\ref{fig:1-2-hcobordism}, where we replace $K$ in the figure with $\Wh^i(K)$. 
\end{enumerate}
\end{lemma}



\begin{figure}[ht]
\centering
\labellist
	\pinlabel {$K$} at 49 80
	\pinlabel {$\langle 0 \rangle$} at 110 130

	\pinlabel {$K$} at 213 80
	\pinlabel {$\langle 0 \rangle$} at 282 130

	\pinlabel {$0$} at 267 80

\endlabellist
\includegraphics[scale=0.9]{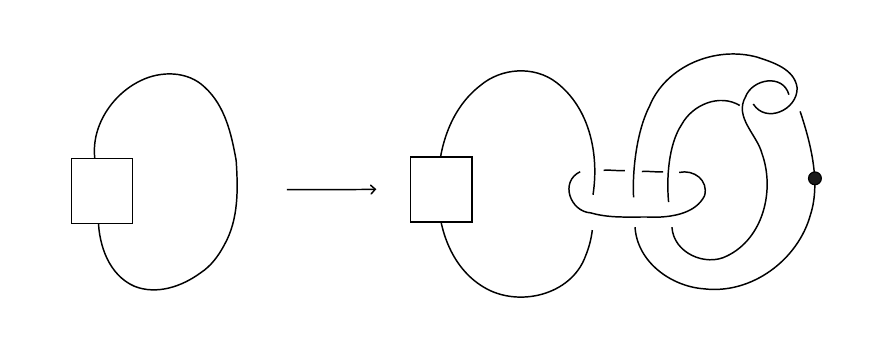}
\caption{A cobordism from $S^3_0(K)$ to $S^3_0(\Wh(K))$ consisting of a $1$- and $2$-handle attachment.}
\label{fig:1-2-hcobordism}
\end{figure}
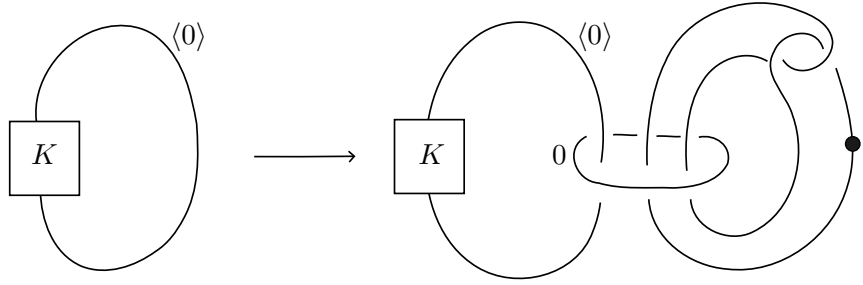

\begin{proof} (\ref{it:sliceR4exhaustion1}):
By Kirby calculus, we can simplify the Casson tower by handle slides and cancellations,
obtaining a description of $X_i$ as the slice disk complement $X_0 = B^4\setminus \nu(D)$ with a copy of $T_1^+$ attached along a 0-framed meridian of $D$, where the dotted circle of $T_1^+$ has been positively Whitehead doubled $i-1$ times.  Since this dotted circle represents the $i$th Whitehead double of a disk $D' \subset X_0$, we have that $X_i$ is the union of $X_0$ and a 2-handle $h$ attached along a 0-framed meridian of $D$, minus the disk $\Wh^i(D')$. Using a parallel copy of $D$ inside $\partial X_0$, we may isotope $D'$ into a parallel copy of $D$ inside a collar of $\partial X_0$. By considering neighborhoods, we may assume this isotopy turns $\Wh^i(D')$ into a copy of $\Wh^i(D)$; deleting $\Wh^i(D)$ and using the 2-handle $h$ to fill the deleted disk $D$ gives the result. The assertion about boundaries follows from the fact that the boundary of any slice disk complement for a knot $K$ is 0-surgery on $K$.

(\ref{it:sliceR4exhaustion2}): By the collapsing step of part (\ref{it:sliceR4exhaustion1}) above, the meridian $\mu$ of the $i$th stage double point loop in $X_i$ is isotopic to a meridian of the deleted disk $\Wh^i(D)$. After isotoping towards the boundary, we see that $\mu$ is isotopic to a meridian of the knot $\Wh^i(K)$ being surgered. By the definition of $X_{i+1}$, the desired cobordism is obtained exactly as in Figure~\ref{fig:1-2-hcobordism} where we replace $K$ in the figure with $\Wh^i(K)$.
\end{proof}

\subsection{End-sums and more exotica}

In the theory of noncompact 4-manifolds a central open question is whether every noncompact 4-manifold admits uncountably many exotic smooth structures. One way of generating candidates for exotic smoothings is to combine a given 4-manifold with an exotic $\bbR^4$ ``at infinity'' by \textit{end-summing}. The \textit{end-sum operation} \cite{gompfinfinite}, or connected sum at infinity, is an operation for combining two noncompact manifolds in a manner similar to the boundary connected-sum of manifolds with boundary. 

\begin{definition}
Let $X, X'$ be noncompact smooth 4-manifolds. The \textit{end-sum} $X\natural X'$ is defined as follows: let $\gamma \colon [0,\infty)\to X$ and $\gamma' : [0,\infty)\to X'$ be smooth, properly embedded rays. Let $\nu(\gamma) \cong \nu(\gamma') \cong [0,\infty)\times B^3$ be regular neighborhoods. Define
\[X\natural X' = X \cup_\phi (I\times \mathbb{R}^3) \cup_\phi' X'\]
where $\phi \colon [0,1/2)\times \bbR^3 \to \nu(\gamma)$ and $\phi' \colon (1/2,1]\times \bbR^3 \to \nu(\gamma')$ are orientation-preserving diffeomorphisms that respect the $\bbR^3$-bundle structures.
\end{definition}

End-summing exotic $\bbR^4$'s onto a noncompact smooth 4-manifold $X$ does not change the homeomorphism type, but it often changes the diffeomorphism type \cite{gompfmenagerie, gompfinfinite, fang1}. In \cite{gompfinfinite} it is shown that for 4-manifolds that are simply connected at infinity, the end-sum does not depend on the choice of rays. In particular, when end-summing exotic $\bbR^4$'s together, one does not need to specify which rays are used. For more general 4-manifolds the choice of rays does matter \cite{endsumnonunique} but this is not relevant to our applications. In particular, end-summing an exotic $\bbR^4$ onto a noncompact 4-manifold $X$ along a specific arc in $X$ still potentially changes the given smoothing of $X$. Since it is known that there are uncountably many distinct exotic $\bbR^4$'s, the challenge is distinguishing between the end-sums $X\natural \caR$ for various exotic $\bbR^4$'s $\caR$. Typically, indirect methods have been used \cite{fang1, gompfmenagerie}; we will distinguish some examples of these smoothings using Gadgil's end Floer homology in the case of $Y\times \bbR$ for closed 3-manifolds $Y$.\\

Let $Y$ be a closed 3-manifold and $\caR$ be a slice $\bbR^4$ constructed as in the previous subsection. Consider $(Y\times \bbR)\natural \caR$ where the end-sum is performed along a standard ray $\{p\}\times [1,\infty)\subset (Y\times\bbR)$ and any ray in $\caR$. The manifold $(Y\times [0,\infty))\natural \caR$ is naturally a submanifold of $(Y\times \bbR)\natural \caR$ and is a neighborhood of the non-standard end. The following lemma, analogous to Lemma \ref{lem:sliceR4exhaustion}, describes a useful compact exhaustion for this end.

\begin{lemma}\label{lem:YxRexhaustion}
The manifold $(Y\times [0,\infty))\natural \caR$ admits a compact exhaustion $X_i$, $i\ge 0$, such that:

\begin{enumerate}
\item  $\partial X_i = -(Y\times\{0\}) \sqcup (Y\# S^3_0(\Wh^i(K)))$ for each $i$.
\item $X_i \subset \mathrm{int} (X_{i+1})$ for each $i$.
\item $\overline{X_{i+1} - X_i}$
 is a cobordism from $Y\#S^3_0(\Wh^i(K))$ to $Y\#S^3_0(\Wh^{i+1}(K)))$ consisting of a 1- and 2-handle attached as in Figure~\ref{fig:1-2-hcobordism} where we replace $K$ in the figure with $\Wh^i(K)$. 
\end{enumerate}
\end{lemma}

\begin{proof}
Note that $(Y\times[0,\infty))\setminus \nu(\{p\}\times [1,\infty)) \cong Y\times[0,1]\cup (Y\setminus B^3)\times[1,\infty)$. The second piece decomposes further as the union of $(Y\setminus B^3)\times[i,i+1]$ for $i \ge 1$. Give $\caR$ the decomposition $X_i$ described in Lemma \ref{lem:sliceR4exhaustion}. Let $\gamma\subset \caR$ be the end-summing ray: since any two proper rays in $\caR$ are isotopic, we may assume $\gamma(0)\in \mathrm{int}(X_0)$, $\gamma(i+1)\in \partial X_i$ for all $i$, $\gamma$ intersects each $\partial X_i$ transversely in a single point, and $\gamma$ is disjoint from the Casson handle used in the construction of $\caR$. By appropriately parametrizing the arc in $Y\times \bbR$ we may assume the end sum $(Y\times [0,\infty))\natural \caR$ equals the union of pieces
\begin{align*}
&K_0:= \left(Y\times[0,1]\cup (Y\setminus B^3)\times[1,2]\right) \cup_{\phi_1} \left(X_0 \setminus \nu(\gamma[0,1])\right)\\
\cup &K_1:= \left((Y\setminus B^3)\times[2,3]\right) \cup_{\phi_2} \left(W_{12} \setminus \nu(\gamma[1,2])\right)\\
\cup & K_2:= \left((Y\setminus B^3)\times[3,4]\right) \cup_{\phi_3} \left(W_{23} \setminus \nu(\gamma[2,3])\right) \cup \hdots
\end{align*}
where $\phi_1$ is a diffeomorphism of $B^3$, $\phi_i$ for $i\ge 1$ are diffeomorphisms of $S^2\times I$, and the $\phi_i$ glue to form a diffeomorphism of $\bbR^3$; the $K_i$ are attached by natural identifications compatible with the $\phi_i$. Letting $X_i = K_0\cup...\cup K_i$ yields the desired exhaustion, the properties follow from the proof of Lemma \ref{lem:sliceR4exhaustion}. 
\end{proof}

%% file: endFloer.tex
\section{End Floer homology}

We briefly recall Gadgil's \cite{gadgil} end Floer homology $\HE(X, \s)$, and give sufficient conditions for when we can define an absolute grading on $\HE(X, \s)$.

Let $X$ be an open 4-manifold with one end, and let 
\[ X_1 \subset X_2 \subset \dots \]
be an exhaustion of $X$ by compact manifolds $X_i$ with $X_i \subset \operatorname{int} X_{i+1}$ and connected boundary. Let $Y_i = \d X_i$, and for $i < j$, let 
\[ W_{ij} = X_j - \operatorname{int} X_i. \]

A cobordism $W_{ij}$ from $Y_i$ to $Y_j$ is \emph{admissible} if the map induced by inclusion 
\[ H^1(W_{ij}) \to H^1(Y_j) \]
is surjective. Admissibility guarantees that $\spinc$ structures $\s_{ij}$ on $W_{ij}$ and $\s_{jk}$ on $W_{jk}$ uniquely determine a $\spinc$ structure on $W_{ik} = W_{ij} \cup W_{jk}$ that restricts to the given ones; see \cite[Section 1]{gadgil} for more details. An \emph{admissible exhaustion} $\{X_i\}$ of a 4-manifold $X$ is an exhaustion of $X$ where each associated $W_{ij}$ is admissible. Throughout, we assume that all of our exhaustions are admissible.

We will be interested in $\spinc$ structures on the ends of 4-manifolds, defined as follows. An \emph{asymptotic spin$^c$ structure} $\s$ on $X$ is a $\spinc$ structure on $X -X'$, where $X'$ is a compact subset of $X$, and two asymptotic $\spinc$ structures $\s'$ on $X - X'$ and $\s''$ on $X-X''$ are equivalent if there exists a compact set $X''' \supset X', X''$ such that $\s' |_{X-X'''} = \s'' |_{X-X'''}$.

With the above set up, Gagdil proves the following:

\begin{theorem}[{\cite[Theorem 1.4]{gadgil}}]
Given an asymptotic spin$^c$ structure $\s$ on $X$, there is an invariant
\[ \HE(X, \s) = \varinjlim \HFred(Y_i, \s|_{Y_i}) \]
which is the direct limit of $\HFred(Y_i, \s|_{Y_i})$ under the morphisms induced by the cobordisms $W_{ij}$. Furthermore, the invariant $\HE(X, \s)$ is independent of the choice of admissibile exhaustion.
\end{theorem}

In certain special cases, we will be interested in endowing $\HE(X, \s)$ with an absolute $\Q$-grading. Note that Gadgil's invariant does not come with a grading, since the cobordism maps induced by $W_{ij}$ in general have a grading shift.

We will require that our exhaustions satisfy the following additional condition:

\begin{definition}\label{def:gradingadmissible}
An admissible exhaustion $\{X_i\}$ on $X$ is \emph{grading admissible} 
if for all $i$, we have $b_2(X_i) = b_3(X_i)  = 0$.  
\end{definition}

We will need to normalize the grading on $\HF^+(Y)$ with respect to $b_1(Y)/2$. Let $\HF^+(Y_i, \s|_{Y_i})[n]$ denote $\HF^+(Y_i, \s|_{Y_i})$ with its grading shifted up by $n$, that is, $\HF^+_k(Y_i, \s|_{Y_i})[n] = \HF^+_{k-n}(Y_i, \s|_{Y_i})$. 

\begin{theorem}\label{thm:absgr}
Suppose that $X$ admits a grading admissible exhaustion and $\s$ is an asymptotic spin$^c$ structure on $X$.  Then the direct limit
\[ \HE(X, \s) = \varinjlim \HFred(Y_i, \s|_{Y_i})[-b_1(Y_i)/2] \]
has an absolute grading that is independent of the choice of grading admissible exhaustion.
\end{theorem}

\begin{remark}
In \cite[Lemma 3.2]{gadgil}, Gadgil needed to choose signs to obtain a direct system; this is not an issue in our case since we work over $\mathbb{F}$.
\end{remark}

\begin{remark}
Since direct limits are unchanged under passing to subsequences, and since any neighborhood of the end of $X$ must contain a subsequence of the $Y_i$, it follows that the absolutely graded $\HE(X,\mathfrak{s})$ is a diffeomorphism invariant of the end of $X$.
\end{remark}

Morally, the idea is that the conditions $b_2(X_i) = b_3(X_i)  = 0$ will imply that the grading shift of a cobordism $W_{ij}$ is governed by the difference between $b_1(Y_i)$ and $b_1(Y_j)$.  Theorem \ref{thm:absgr} will follow readily from the following proposition:

\begin{proposition}\label{prop:grshift}
Suppose that $\{X_i\}$ is a grading admissible exhaustion of $X$ with respect to an asymptotic spin$^c$ structure $\s$. 
Then
\[ \gr(F_{W_{ij}, \s |_{W_{ij}}}(x)) - \gr(x) = \frac{b_1(Y_j) - b_1(Y_i)}{2} \]
where $x \in \HF^+(Y_i, \s|_{Y_i})$ is a homogenously graded element.
\end{proposition}

In order to prove Proposition \ref{prop:grshift}, we will need the following lemma:
\begin{lemma}\label{lem:b1XXYY}
Let $\{X_i\}$ be a grading admissible exhaustion of $X$. Then
\begin{equation*}
	b_1(X_i) = b_1(Y_i).
\end{equation*}
\end{lemma}

\begin{proof}
Consider the long exact sequence of the pair $(X_i, Y_i)$ with $\Q$-coefficients:
\begin{equation}\label{eq:lesXY}
	 \cdots \to H_3(X_i) \to H_3(X_i, Y_i) \to H_2(Y_i) \to H_2(X_i) \to \cdots
\end{equation}
Since $\{X_i\}$ is grading admissible, we have that $H_3(X_i) = H_2(X_i) = 0$. Then
\begin{align*}
	H^1(Y_i) &= H_2(Y_i) \\
		&= H_3(X_i, Y_i) \\
		&= H^1(X_i)
\end{align*}
where the first equality follows from Poincar\'e duality, the second from exactness of \eqref{eq:lesXY}, and the third from Poincar\'e-Lefschetz duality.
\end{proof}

With this lemma in hand, we can now prove Proposition \ref{prop:grshift}.
\begin{proof}[Proof of Proposition \ref{prop:grshift}]
We have that
\begin{align*}
	 \gr(F_{W_{ij},\s}(x)) - \gr(x) &= \frac{c_1(\s)^2-3\sigma(W_{ij}) -2 \chi(W_{ij})}{4} \\
	 					&= \frac{-\chi(W_{ij})}{2} \\
						&= \frac{b_1(X_j) - b_1(X_i)}{2} \\
						&=  \frac{b_1(Y_j) - b_1(Y_i)}{2} 
\end{align*}
where the first equality is \cite[Theorem 7.1]{OS-triangles}, the second and third equalities follows from the fact that $b_2(X_i) = b_3(X_i)  = 0$ for all $i$, and the fourth equality follows from Lemma \ref{lem:b1XXYY}.
\end{proof}

We are now ready to prove Theorem \ref{thm:absgr}.

\begin{proof}[Proof of Theorem \ref{thm:absgr}]
Let $\{X_i \}$ be a grading admissible exhaustion with respect to an asymptotic $\spinc$ structure $\s$. We first show that there exists an absolute grading on the direct limit. Recall that 
\[ \varinjlim \HFred(Y_i, \s|_{Y_i})[-b_1(Y_i)/2] = \bigsqcup_i \HFred(Y_i, \s|_{Y_i})[-b_1(Y_i)/2] / \sim \]
where for $x_i \in \HFred(Y_i, \s|_{Y_i})$ and $x_j \in \HFred(Y_j, \s|_{Y_j})$, we have that $x_i \sim x_j$ if there exists some $k$ such that $F_{W_{ik}, \s|_{W_{ik}}}(x_i) = F_{W_{jk}, \s|_{W_{jk}}}(x_j)$. Let $\gr$ denote the absolute grading on $\HFred(Y_i, \s|_{Y_i})$ and $\gr_\sh$ the absolute grading on $\HFred(Y_i, \s|_{Y_i})[-b_1(Y_i)/2]$; that is,
\[ \gr_\sh(x_i) = \gr(x_i) - \frac{b_1(Y_i)}{2}. \]
Then
\begin{align*}
 	\gr_\sh(x_i) &= \gr(x_i) - \frac{b_1(Y_i)}{2} \\
		&= \gr(F_{W_{ik}, \s|_{W_{ik}}}(x_i)) -\frac{b_1(Y_k)}{2} \\
		&= \gr(F_{W_{jk}, \s|_{W_{jk}}}(x_j)) -\frac{b_1(Y_k)}{2} \\
		&= \gr(x_j) - \frac{b_1(Y_j)}{2} \\
		&= \gr_\sh(x_j)
\end{align*}
where the second and fourth equalities follow from Proposition \ref{prop:grshift}. Thus, there exists an absolute grading on the direct limit.

We now show that the absolute grading is independent of the choice of grading admissible exhaustion. The proof is identical to \cite[Proposition 3.4]{gadgil}; for completeness, we recount it here. First, note that the absolute grading is unaffected by passing to subsequences.  Let $\{X_i\}$ and $\{X'_i\}$ be two grading admissible exhaustions of $X$ with respect to an asymptotic $\spinc$ structure $\s$. By passing to subsequences, we may assume that
\[ X_1 \subset X'_1 \subset X_2 \subset X'_2 \subset \dots \]
and \cite[Lemma 2.5]{gadgil} implies that the exhaustion 
\begin{equation}\label{eq:exhaust}
	X'_1 \subset X_2 \subset X'_2 \subset \dots
\end{equation}
is admissible. Furthermore, this exhaustion is grading admissible, since $\{X_i\}$ and $\{X'_i\}$ are grading admissible. Since $X_2 \subset X_3 \subset \dots$ and $X'_1 \subset X'_2 \subset \dots$ are subsequences of the exhaustion in \eqref{eq:exhaust}, it follows that the direct limits with respect to $\{X_i\}$ and $\{X'_i\}$ agree, as desired.
\end{proof}

%% file: HFprelim.tex
\section{Heegaard Floer preliminaries}

We assume the reader is familiar with Heegaard Floer homology \cite{Os-3mfds, OS-abs}, in particular knot Floer homology \cite{OS-knots, Rasmussen-PhD} and the mapping cone formula \cite{OS-integer}. We briefly recall these ideas, primarily to establish notation.

\subsection{The knot Floer complex}
To a knot $K \subset S^3$, we consider the associated chain complex $\CFKm(K)$ that is freely and finitely generated over $\F[U]$. Each generator has a Maslov, or homological, grading, as well as an Alexander grading. The Alexander grading of the generators induces a filtration on $\CFKm(K)$, and the filtered chain homotopy type of $\CFKm(K)$ is a knot invariant.

As an $\F$-vector space, $\CFKm(K) = \oplus_{i,j} \CFKm(K, i, j)$, where $i \in \Z_{\leq 0}$ is the negative of the $U$-exponent and $j\in \Z$ is the Alexander grading. It is common to depict $\CFKm(K)$ in the $(i,j)$-plane. That is, we place the generators along the vertical axis, according to their Alexander grading. (The Maslov grading is suppressed from this picture.) The variable $U$ shifts a generator one unit down and one unit to the left. See Figure \ref{fig:CFK41} for an example. We will also consider the complex  $\CFKi(K) = \CFKm(K) \otimes_{\F[U]} \F[U, U^{-1}]$.  We have a similar decomposition $\CFKi(K) = \oplus_{i,j \in \Z} CFK^\infty(K,i,j)$.

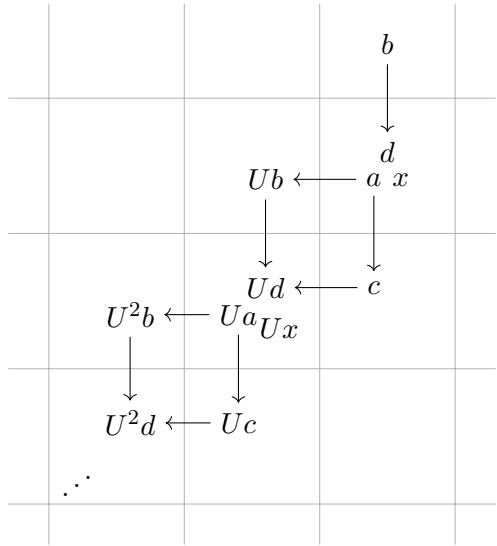
\begin{figure}[ht]
\begin{tikzpicture}[scale=1.8]
	\draw[step=1, black!30!white, very thin] (-2.3, -2.3) grid (1.3, 1.7);

	\node at (0.5, 1.4) (b) {$b$};
	\node at (0.5, 0.6) (d) {$d$};

	\node at (0.6, 0.4) (x) {$x$};

	\node at (0.4, 0.4) (a) {$a$};
	\node at (-0.4, -0.4) (Ud) {$Ud$};
	\node at (-0.4, 0.4) (Ub) {$Ub$};
	\node at (0.4, -0.4) (c) {$c$};
	
	\node at (-0.3, -0.7) (Ux) {$Ux$};

	\node at (-0.6, -0.6) (Ua) {$Ua$};
	\node at (-1.4, -1.4) (U2d) {$U^2d$};
	\node at (-1.4, -0.6) (U2b) {$U^2b$};
	\node at (-0.6, -1.4) (Uc) {$Uc$};	

	\draw[->] (b) to node[]{} (d);

	\draw[->] (a) to node[]{} (Ub);
	\draw[->] (a) to node[]{} (c);
	\draw[->] (Ub) to node[]{} (Ud);
	\draw[->] (c) to node[]{} (Ud);

	\draw[->] (Ua) to node[]{} (U2b);
	\draw[->] (Ua) to node[]{} (Uc);
	\draw[->] (U2b) to node[]{} (U2d);
	\draw[->] (Uc) to node[]{} (U2d);

	\node at (-1.8, -1.8) (dots) {$\iddots$};
	
\end{tikzpicture}
\caption{The knot Floer complex $\CFKm(4_1)$.}
\label{fig:CFK41}
\end{figure}

We will also be interested in null-homologous knots $J$ in a $3$-manifold $Y$, where $Y$ is a homology $S^1 \times S^2$. For the torsion $\spinc$ structure $\s_0$ on $Y$, we obtain an analogous picture for $\CFKi(Y,J, \s_0)$ as for knots in $S^3$.

\subsection{The mapping cone formula}
We now recall the mapping cone formula for determining $\HF^+(S^3_n(K))$, following \cite{OS-integer}. The formula more generally holds for null-homologous knots in a general $3$-manifold $Y$, as long as we restrict to torsion $\spinc$ structures.

Let $C=\CFKi(Y,K, \s_0)$ where $\s_0$ is a torsion $\spinc$ structure on $Y$ and recall that as a vector space, we have the decomposition $C=\bigoplus_{i,j \in \Z} C(i,j)$.  For a subset $X$ of $\mathbb{Z}^2$, let $CX = \bigoplus_{(i,j) \in X} C(i,j)$.
We have the following quotient complexes 
\begin{align*}
	A_s &= C\{ \max(i, j-s) \geq 0 \} \\
	B &= C\{ i \geq 0 \}.
\end{align*}
We also have maps
\begin{align*}
	v_s \colon A_s \to B \\
	h_s \colon A_s \to B.
\end{align*}
The map $v_s$ is given by projection. The map $h_s$ is projection onto $C\{ j \geq s \}$ followed by a chain homotopy equivalence between $C\{ j \geq s \}$ and $B$.  For more details on the chain homotopy equivalence, see \cite{OS-integer}.  Consider the map
\[ \Phi_n \colon \bigoplus_{s \in \Z} A_s \to \bigoplus_{s \in \Z} B_s \]
where each $B_s$ is just a copy of $B$ and 
\[ \Phi_n = v_s + h_s \]
where
\begin{align*}
	v_s &\colon A_s \to B_s \\
	h_s &\colon A_s \to B_{s+n}.
\end{align*}
Then the main result of \cite{OS-integer} states that 
\[ H_*(\Cone \Phi_n) = \HF^+(Y_n(K)), [\s_0[) \]
where $[\s_0]$ denotes summing over the $\spinc$ structures on $Y_n(K)$ which are cobordant to $\s_0$ across the associated 2-handle cobordism.
Moreover, $\Cone \Phi_n$ inherits a relative $\Z$-grading from the gradings on $A_s$ and $B$, together with the fact that in the complex $\Cone \Phi_n$, the map $\Phi_n$ lowers grading by one. This relative grading can be lifted to an absolute $\Q$-grading which recovers the absolute $\Q$-grading on $HF^+(Y_n(K)), [\s_0])$, restricted to the appropriate $\spinc$ structures. We will be interested in the case $n=-1$, in which case the absolute grading on $\Cone \Phi_n$ is given by just keeping the grading that $B_0$ inherits from $\CFKi$, and in the case $n=0$, where the absolute grading on $\Cone \Phi_n$ is given by shifting the grading on $B_0$ down by $1/2$.

Lastly, we observe that in the case when $\HFred(Y)=0$, we may pass to homology (i.e., consider $v_{s,*} \colon H_*(A_s) \to H_*(B)$ instead of $v_s$, and similarly for $h_{s,*}$) before taking the mapping cone.

%% file: cobordismmap.tex
\section{The cobordism map between $0$-surgeries}
\subsection{A key theorem}
The goal of this section is to prove the following theorem about cobordism maps between $0$-surgeries along iterated Whitehead doubles. Note that in the statement of the theorem, the knot $K$ is already a Whitehead double.

\begin{theorem}\label{thm:0-surgerycobordism}
Let $K \subset S^3$ be the Whitehead double of a slice knot. 
\begin{enumerate}
	\item \label{it:0-surgerycobordism1} We have
		\[
			\HF^+(S^3_0(\Wh(K))) =  \cT^+_{(1/2)} \oplus \cT^+_{(-1/2)} \oplus \Big( \HFred (S^3_0(K)) \otimes V \Big)
		\]
		where $V = \F^2_{(0)} \oplus \F^2_{(-1)}$. 
	\item \label{it:0-surgerycobordism2} The cobordism map from $\HF^+(S^3_0(K)) \to \HF^+(S^3_0(\Wh(K)))$ given by the $1$- and $2$-handle attachment in Figure \ref{fig:1-2-hcobordism} is injective on the highest graded summand of $\HFred(S^3_0(K))$.
\end{enumerate}
\end{theorem}

%
%
%

\begin{remark}\label{rem:0-surgerycobordism}
Theorem \ref{thm:0-surgerycobordism} holds more generally whenever $K$ is a genus one knot whose knot Floer complex $\CFKm(K)$ is a direct sum of a single generator and any number of $1 \times 1$ boxes, as in Figure \ref{fig:box}.
\end{remark}

The strategy for proving Theorem \ref{thm:0-surgerycobordism} is as follows. First, the cobordism map for the $1$-handle is described in \cite{OS-triangles}.  More precisely, there is a canonical isomorphism  
\[
HF^+(S^3_0(K) \# S^1 \times S^2) \cong HF^+(S^3_0(K)) \otimes_{\mathbb{F}[U]}  \left(\mathbb{F}[U]_{(-1/2)} \oplus \mathbb{F}[U]_{(1/2)} \right).
\]
The map for the 1-handle cobordism includes $HF^+(S^3_0(K))$ into the copy with the grading shift of $+1/2$. 
The difficulty then lies in determining the map $F$ for the $2$-handle cobordism.  We determine the map $F$ using the surgery exact triangle in Figure \ref{fig:triangle}. Indeed, our strategy is to compute $\HF^+$ of each of the three 3-manifolds, and then it turns out that a grading and exactness argument completes the proof.

\begin{figure}[ht]
\centering
\labellist
	\pinlabel {$0$} at 30 243
	\pinlabel {$K$} at 30 203
	\pinlabel {$\infty$} at 80 203
	\pinlabel {$0$} at 215 243
	
	\pinlabel {$0$} at 298 243
	\pinlabel {$K$} at 298 203
	\pinlabel {$0$} at 348 203
	\pinlabel {$0$} at 483 243
	
	\pinlabel {$0$} at 160 98
	\pinlabel {$K$} at 160 58
	\pinlabel {$1$} at 210 58
	\pinlabel {$0$} at 343 98
	
	\pinlabel {$F$} at 250 212	
	\pinlabel {$\Phi$} at 380 125	
	\pinlabel {$\Psi$} at 125 120	

\endlabellist
\includegraphics[scale=0.75]{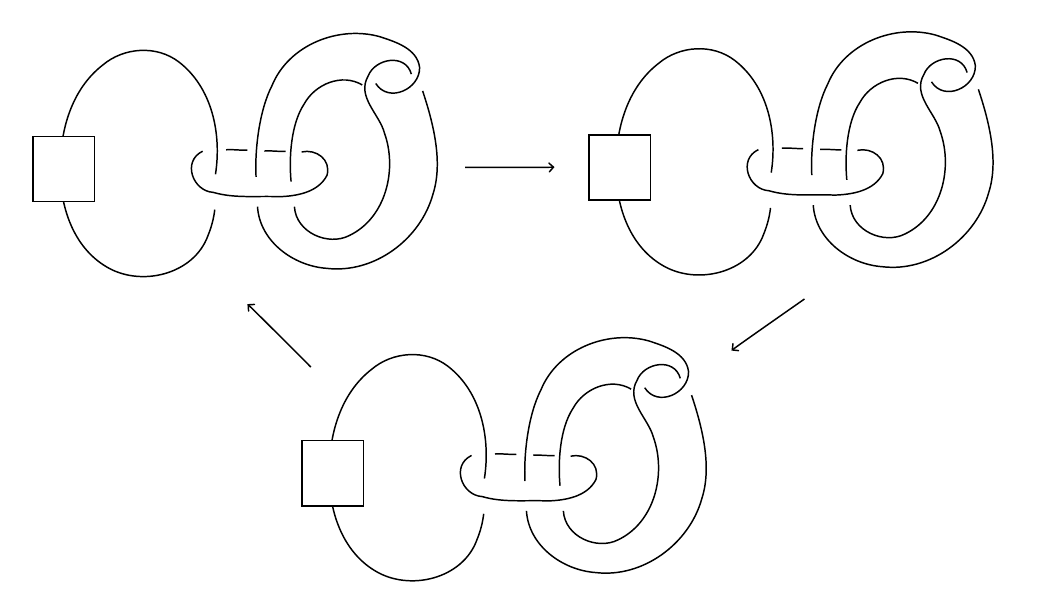}
\caption{The manifolds involved in the surgery exact triangle. (We mildly abuse notation and use $F, \Phi$, and $\Psi$ to refer to both the cobordisms and their induced maps on $\HF^+$.)}
\label{fig:triangle}
\end{figure}
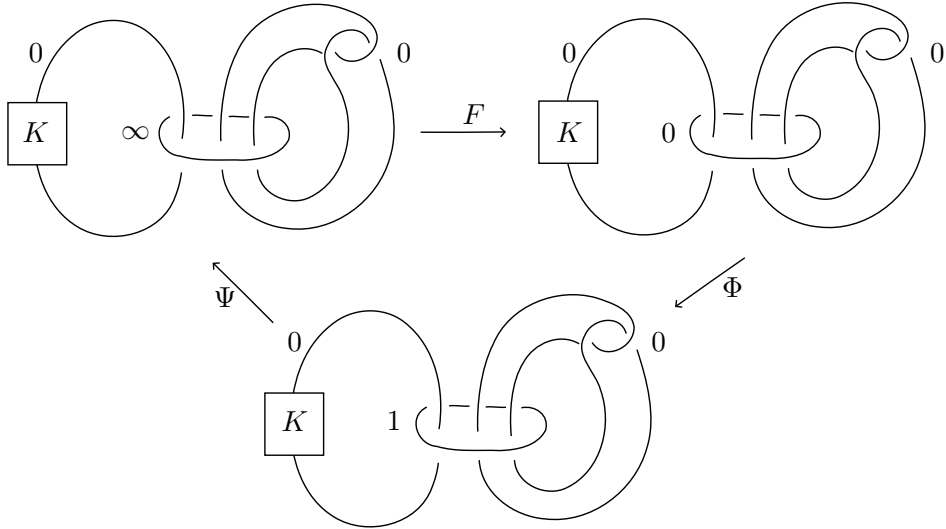

Computing $\HF^+$ for the top two 3-manifolds in Figure \ref{fig:triangle} is straightforward, as the manifolds are $S^3_0(K)$ and $S^3_0(\Wh(K))$ respectively, and Hedden has given a formula for the knot Floer homology of Whitehead doubles \cite{Hedden-Whitehead}. Computing $\HF^+$ for the bottom 3-manifold is slightly more involved. Blowing down the $+1$-framed unknot yields the Kirby diagram in Figure \ref{fig:KJinY}, which is the connected sum of $K \subset S^3$ with the knot $J \subset Y$ in Figure \ref{fig:KirbyY}. That is, $J$ is a null-homologous knot in $0$-surgery on the right-handed trefoil.  If we can compute the knot Floer complex of $J \# K$, then we can apply the mapping cone formula to compute the Floer homology of this remaining 3-manifold, $Y_{-1}(J \# K)$.

\begin{figure}[ht]
\centering
\labellist
	\pinlabel {$K$} at 31 82
	\pinlabel {$0$} at 155 150
	\pinlabel {$-1$} at 134 72
	\pinlabel {$-1$} at 38 109
\endlabellist
\includegraphics[scale=0.9]{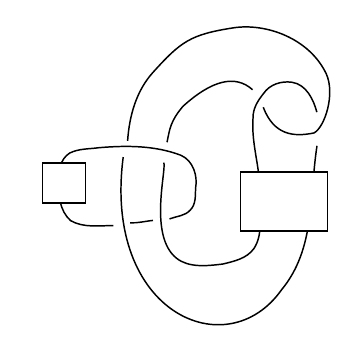}
\caption{}
\label{fig:KJinY}
\end{figure}

\begin{figure}[ht]
\centering
\labellist
	\pinlabel {$J$} at 8 75
	\pinlabel {$0$} at 140 150
	\pinlabel {$-1$} at 115 72
\endlabellist
\includegraphics[scale=0.9]{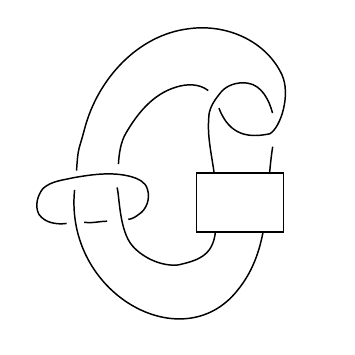}
\caption{The knot $J$ in $Y$.}
\label{fig:KirbyY}
\end{figure}

Since we have the K\"unneth formula at hand and we know $\CFKm(S^3, K)$, it is sufficient to compute the knot Floer complex $\CFKm(Y,J)$. We carry this computation out in Lemma \ref{lem:JYcomp} by comparing the $d$-invariants and $\HFred$ for $Y$ to those of $Y_{-1}(J)$. The key observation is that $J$ is a knot in a manifold with trivial $\HFred$, and $-1$-surgery along $J$ also has trivial $\HFred$. Then familiar results about (a mild generalization of) L-space knots in L-spaces yield the desired formula for the Floer homology of $Y_{-1}(J \# K)$, needed for the exact triangle.  That is the subject of Lemma~\ref{lem:HFcomp}.  

We now proceed with this strategy for proving Theorem \ref{thm:0-surgerycobordism}.

\begin{lemma}\label{lem:JYcomp}
The knot Floer complex $\CFKm(Y,J)$ is generated over $\F[U]$ by $a,b,c,$ and $d$ with the following gradings and boundary map:
\begin{center}
\begin{tabular}{*{10}{@{\hspace{10pt}}c}}
\hline
& && $m$ && $A$ && $\d$  & \\
\hline
& $a$ && $1/2$ && $1$ && $b$\\ 
& $b$ && $-1/2$ && $0$ && $0$\\ 
& $c$ && $-3/2$ && $-1$ && $Ub$\\ 
& $d$ && $-1/2$ && $0$ && $0$\\ 
\hline
\end{tabular}
\vspace{5pt}
\end{center}
which we depict graphically in Figure \ref{fig:JY}.
\end{lemma}

\begin{figure}[ht]
\begin{tikzpicture}[scale=1.25]
	\node at (0, 1) (a) {\lab{a}};
	\node at (0, 0) (b) {\lab{b}};
	\node at (-1, -1) (Ub) {\lab{Ub}};
	\node at (-1, 0) (Ua) {\lab{Ua}};
	\node at (0, -1) (c) {\lab{c}};
	\node at (-0.2, -0.2) (d) {\lab{d}};
	\node at (-1.2, -1.2) (Ud) {\lab{Ud}};

	\draw[->] (a) to node[]{} (b);
	\draw[->] (Ua) to node[]{} (Ub);
	\draw[->] (c) to node[]{} (Ub);

	\node at (-1.6, -1.3) (dots) {$\iddots$};
	
\end{tikzpicture}
\caption{The knot Floer complex of $J \subset Y$.}
\label{fig:JY}
\end{figure}

\begin{proof}
Recall that $Y$ is $S^3_0(T_{2,3})$. A straightforward calculation using the mapping cone formula shows that 
\[ \HF^+(Y) = \cT^+_{(-3/2)} \oplus \cT^+_{(-1/2)}. \] 
In particular, $\HFred(Y)=0$. Furthermore, $-1$-surgery on $J$ yields $S^1 \times S^2$, and so
\[ \HF^+(Y_{-1}(J)) = \HF^+(S^3_0(U)) = \cT^+_{(1/2)} \oplus \cT^+_{(-1/2)}. \]

We observe that $J$ is a genus one knot in a manifold $Y$ with $\HFred(Y)=0$, and that $\HFred(Y_{-1}(J))$ is also zero. Thus, we are in a situation similar to that of a negative L-space knot in $S^3$. We now use $d$-invariants. First, consider the submodules of $\HF^+(Y)$ and $\HF^+(Y_{-1}(J))$ that are in the image of the $H_1$-action. Since $d_{-1/2}(Y)=d_{-1/2}(Y_{-1}(J))=-1/2$, we conclude that $\CFKm(Y, J)$ has a subcomplex consisting of a single generator $d$ in Maslov grading $-1/2$ and Alexander grading 0; that is, up to a grading shift, its knot Floer complex looks like that of the unknot in $S^3$.

We now consider the cokernel of the $H_1$-action. We have that 
\begin{itemize}
	\item $d_{1/2}(Y)=-3/2$,  
	\item $d_{1/2}(Y_{-1}(J))=1/2$, and
	\item $\HFred(Y)=\HFred(Y_{-1}(J))=0$.
\end{itemize}
By Proposition 1.1 of \cite{Krcatovich2018} (which is a consequence of \cite[Theorem 1.2]{OS-lens}), knot Floer complexes of L-space knots in $S^3$ have a particularly special form: their knot Floer complexes look like staircases. This result holds more generally for L-space knots in any integer homology sphere L-space (the proof is identical), and by mirroring, we recover a result for negative L-space knots. In particular, the unique knot Floer complex for a genus one knot on which $-1$-surgery yields a manifold with $\HFred=0$ is, up to a grading shift, the knot Floer complex of the left-handed trefoil, that is, the knot Floer complex generated by $a, b,$ and $c$ with gradings and differential as in the statement of Lemma \ref{lem:JYcomp}.

Having identified the associated graded complexes of the 2-step filtration given by the $H_1$-action, it is now straightforward to verify that, up to a change of basis, $\CFKm(Y, J)$ is as stated.  \end{proof}

Throughout, let $B[k_i]$ be the $1 \times 1$-box generated by $a, b, c, d$ with gradings and differential:
\begin{center}
\begin{tabular}{*{10}{@{\hspace{10pt}}c}}
\hline
& && $m$ && $A$ && $\d$  & \\
\hline
& $a$ && $k_i$ && $0$ && $Ub+c$\\ 
& $b$ && $k_i+1$ && $1$ && $d$\\ 
& $c$ && $k_i-1$ && $-1$ && $Ud$\\ 
& $d$ && $k_i$ && $0$ && $0$\\ 
\hline
\end{tabular}
\vspace{5pt}
\end{center}
We depict this box graphically in Figure \ref{fig:box}.

\begin{figure}[ht]
\begin{tikzpicture}[scale=2]
	\node at (0, 0) (a) {$a_{(k_i)}$};
	\node at (-1, -1) (Ud) {$Ud_{(k_i-2)}$};
	\node at (-1, 0) (Ub) {$Ub_{(k_i-1)}$};
	\node at (0, -1) (c) {$c_{(k_i-1)}$};

	\draw[->] (a) to node[]{} (Ub);
	\draw[->] (a) to node[]{} (c);
	\draw[->] (Ub) to node[]{} (Ud);
	\draw[->] (c) to node[]{} (Ud);
	
\end{tikzpicture}
\caption{A $1 \times 1$-box $B[k_i]$, with the Maslov gradings denoted by the subscript. For example, the grading of $Ub$ is $k_i-1$.}
\label{fig:box}
\end{figure}

With $\CFKm(Y,J)$ in hand, we are ready to compute the Heegaard Floer homologies of the manifolds in Figure \ref{fig:triangle}.

\begin{lemma}\label{lem:HFcomp}
Let $K \subset S^3$ be a genus one knot whose knot Floer complex $\CFKm(K)$ is of the form $x \oplus \bigoplus_{i=1}^n B[k_i]$, where $x$ is in Maslov and Alexander grading $0$ and $B[k_i]$ is the $1 \times 1$-box described above.
Then
\begin{enumerate}
	\item\label{it:HFcomp1} $\HF^+(S^3_0(K) \# S^1 \times S^2) =  \cT^+_{(1)} \oplus (\cT^+_{(0)})^2 \oplus \cT^+_{(-1)} \oplus  \bigoplus_{i=1}^n \big(\F_{(k_i)} \oplus \F_{(k_i-1)})\big)$
	\item\label{it:HFcomp2} $\HF^+(S^3_0(\Wh(K))) = \cT^+_{(1/2)} \oplus \cT^+_{(-1/2)} \oplus  \bigoplus_{i=1}^n \big( \F_{(k_i-\frac{1}{2})}^2 \oplus \F_{(k_i-\frac{3}{2})}^2 \big) $
	\item\label{it:HFcomp3} $\HF^+(Y_{-1}(J \# K)) = \cT^+_{(1/2)} \oplus \cT^+_{(-1/2)} \oplus  \bigoplus_{i=1}^n \big( \F_{(k_i-\frac{1}{2})}^2 \oplus \F_{(k_i-\frac{3}{2})}^2 \big) $.
\end{enumerate}
\end{lemma}

\begin{proof}
We begin by proving \eqref{it:HFcomp1}. We first compute $\HF^+(S^3_0(K))$. Since $K$ is genus one, we need only consider the homology of the mapping cone of
\[
\begin{tikzcd}[column sep=.5cm, row sep=1cm]
A_0 \ar[d, "v_0+h_0", bend right=0, labels=right] \\
B_0.
\end{tikzcd}
\]
We have that
\[ H_*(A_0) =\cT^+_{(0)} \oplus \bigoplus_{i=1}^n \F_{(k_i-1)}  \]
and
\[ H_*(B_0) = \cT^+_{(0)}, \]
and that $v_{0,*}$ and $h_{0,*}$ are both isomorphisms on $\cT^+$ and zero on $\bigoplus_{i=1}^n \F_{(k_i-1)}$. The grading shift in the mapping cone of $v_0+h_0$ lowers the grading of $H_*(B_0)$ by $1/2$ and raises the grading of $H_*(A_0)$ by $1/2$. Applying the K\"unneth formula for connect summing with $S^1 \times S^2$ yields the desired result.

Next we prove \eqref{it:HFcomp2}. Proposition~\ref{prop:WhiteheadK} below shows that $\CFKm(\Wh(K))$ is of the form
\[  x \oplus \bigoplus_{i=1}^n \big( B[k_i]^2 \oplus B[k_i-1]^2 \big), \]
where $x$ is in Maslov and Alexander grading $0$. The result now follows as in part \eqref{it:HFcomp1}.

Lastly, we prove \eqref{it:HFcomp3}. We begin by applying the K\"unneth formula to determine $\CFKm(Y, J \# K)$, which we see is of the form
\[ \CFKm(Y, J) \oplus \bigoplus_{i=1}^n \Big( B[k_i-\tfrac{1}{2} ]^2 \oplus B[k_i+\tfrac{1}{2} , 1] \oplus B[k_i-\tfrac{3}{2}, -1] \Big) \]
where $B[k_i, j]$ denotes the $1 \times 1$-box generated by $a, b, c, d$ with gradings and differential:
\begin{center}
\begin{tabular}{*{10}{@{\hspace{10pt}}c}}
\hline
& && $m$ && $A$ && $\d$  & \\
\hline
& $a$ && $k_i$ && $j$ && $Ub+c$\\ 
& $b$ && $k_i+1$ && $j+1$ && $d$\\ 
& $c$ && $k_i-1$ && $j-1$ && $Ud$\\ 
& $d$ && $k_i$ && $j$ && $0$\\ 
\hline
\end{tabular}
\vspace{5pt}
\end{center}

\noindent Note that $B[k_i, 0]$ is exactly $B[k_i]$. We now consider the mapping cone formula for $Y_{-1}(J \# K)$. Since the genus of $J \# K \subset Y$ is two, the truncated mapping cone takes the form
\[
\begin{tikzcd}[column sep=1cm, row sep=1cm]
& A_{-1} \ar[dl, "h_{-1}", labels=right]  \ar[d, "v_{-1}", labels=right] & A_0 \ar[dl, "h_0", labels=right]  \ar[d, "v_0", labels=right] & A_1 \ar[dl, "h_1", labels=right]  \ar[d, "v_1", labels=right]\\
B_{-2} & B_{-1} & B_0 & B_1.
\end{tikzcd}
\]
We know that the $\CFKm(Y,J)$ summand of $\CFKm(Y, J \#K)$ contributes $\cT^+_{1/2} \oplus \cT^+_{-1/2}$ to $\HF^+(Y_{-1}(J \# K))$, since $Y_{-1}(J) = S^1 \times S^2$.  (In general, a summand of $CFK^-(Y,J)$ does not necessarily split off as a summand in the mapping cone formula, but it does in the case of a 3-manifold with $b_1 \leq 1$ and $HF_{red} = 0$.)  So we are left with determining the contribution of 
\[ C=  \bigoplus_{i=1}^n \Big( B[k_i-\tfrac{1}{2} ]^2 \oplus B[k_i+\tfrac{1}{2} , 1] \oplus B[k_i-\tfrac{3}{2}, -1] \Big). \]
Focusing on the $A_s$ complexes associated to $C$, we have 
\begin{align*}
	H_*(A_{-1}) &= \bigoplus_{i=1}^n \F_{(k_i-\tfrac{5}{2})} \\
	H_*(A_{0}) &= \bigoplus_{i=1}^n \F^2_{(k_i-\tfrac{3}{2})} \\
	H_*(A_{1}) &= \bigoplus_{i=1}^n \F_{(k_i-\tfrac{1}{2})}.
\end{align*}
The maps $v_{i,*}$ and $h_{i,*}$ are all zero on $H_*(A_s)$, $s=-1, 0,$ and $1$. For $-1$-surgery, the grading shift in the mapping cone formula raises the grading of $A_{-1}$ by 1,  raises the grading of $A_0$ by 1, and lowers the grading of $A_1$ by 1, giving the result as stated.
\end{proof}

Now that we know the Heegaard Floer homology of each of the manifolds in Figure \ref{fig:triangle}, we use gradings and exactness to prove the following lemma:

\begin{lemma}\label{lem:gradinginjection}
Let $K \subset S^3$ be a knot whose knot Floer complex is of the form $x \oplus \bigoplus_{i=1}^n B[k_i]$, where $x$ is in Maslov and Alexander grading $0$ and $B[k_i]$ is the $1 \times 1$-box. Then the 2-handle cobordism map
\[ F \colon \HF^+(S^3_0(K) \# S^1 \times S^2) \to \HF^+(S^3_0(\Wh(K))) \]
as in Figure \ref{fig:triangle} is injective on the highest graded summand of $\HFred(S^3_0(K) \# S^1 \times S^2)$.  
\end{lemma}

\begin{proof}
The proof relies on the following surgery exact triangle:

\begin{equation*}
\begin{tikzpicture}[baseline=(current  bounding  box.center)]
\node(1)at (-3,2){$\HF^+(S^3_0(K) \# S^1 \times S^2)$};
\node(2)at (3,2){$\HF^+(S^3_0(\Wh(K)))$};
\node(3)at(0,0){$\HF^+(Y_{-1}(J \# K))$,};
\path[->](1)edge node[above]{$F$}(2);
\path[->](2)edge node[right]{$\Phi$}(3);
\path[->](3)edge node[left]{$\Psi$}(1);
\end{tikzpicture}
\end{equation*}
where $J \subset Y$ is as in Figure \ref{fig:KirbyY}.

We computed the values of these Heegaard Floer homologies in Lemma \ref{lem:HFcomp}. Recall from \cite{OS-abs} that $F$ and $\Psi$ lower grading by $1/2$.  Indeed, the maps in the exact triangle consist a priori of the sum of cobordism maps over all spin$^c$ structures.  However, the Floer homology is supported in the unique torsion spin$^c$ structure for each 3-manifolds in the exact triangle, so the only potential spin$^c$ cobordism map that could contribute to $F$ (respectively $\Psi$) is the one associated to the unique spin$^c$ structure with trivial first Chern class.  The grading shift follows. We have that $\Psi$ maps the two copies of $\cT^+$ in $\HF^+(Y_{-1}(J \# K)$ injectively into $\HF^+(S^3_0(K) \# S^1 \times S^2)$. Similarly, restricting to the submodule of elements in the image of arbitrarily high powers of $U$, we have that $F$ is injective on the cokernel of $\Psi$. In particular, it suffices to consider the exact triangle on $\HFred$:
\begin{equation}\label{eq:exactred}
\begin{tikzpicture}[baseline=(current  bounding  box.center)]
\node(1)at (-3,2){$\bigoplus_{i=1}^n \big(\F_{(k_i)} \oplus \F_{(k_i-1)})\big)$};
\node(2)at (3,2){$\bigoplus_{i=1}^n \big( \F_{(k_i-\frac{1}{2})}^2 \oplus \F_{(k_i-\frac{3}{2})}^2 \big) $};
\node(3)at(0,0){$ \bigoplus_{i=1}^n \big( \F_{(k_i-\frac{1}{2})}^2 \oplus \F_{(k_i-\frac{3}{2})}^2 \big) $.};
\path[->](1)edge node[above]{$F$}(2);
\path[->](2)edge node[right]{$\Phi$}(3);
\path[->](3)edge node[left]{$\Psi$}(1);
\end{tikzpicture}
\end{equation}

Since $\Psi$ lowers grading by $1/2$, it follows that the highest graded summand of $\HFred(S^3_0(K) \# S^1 \times S^2)$ is not in the image of $\Psi$. The desired result now follows from exactness.
\end{proof}

Finally, we are ready to prove Theorem \ref{thm:0-surgerycobordism}.

\begin{proof}[Proof of Theorem \ref{thm:0-surgerycobordism}]
Item \eqref{it:0-surgerycobordism1} of the theorem is \eqref{it:HFcomp2} of Lemma \ref{lem:HFcomp}, combined with Proposition~\ref{prop:WhiteheadK}.  Item \eqref{it:0-surgerycobordism2} of the theorem follows by composing the 1-handle cobordism map with the map $F$ in Lemma \ref{lem:gradinginjection}.
\end{proof}

\subsection{Negatively clasped Whitehead doubles}
A natural question to ask is whether we can obtain a version of Theorem \ref{thm:0-surgerycobordism} \eqref{it:0-surgerycobordism2} for the \emph{negatively} clasped Whitehead double $\Wh^-$. Indeed, such a result would help us to study Casson handles with mixed signs. 

Unfortunately, the version of Theorem \ref{thm:0-surgerycobordism} \eqref{it:0-surgerycobordism2} for the negatively clasped Whitehead double is that the cobordism map from $\HF^+(S^3_0(K))$ to $\HF^+(S^3_0(\Wh^-(K)))$ is the zero map.

\begin{theorem}\label{thm:negativedouble}
If $K$ is the negative Whitehead double of a nontrivial slice knot, then the cobordism map
\[\HFred(S^3_0(K)) \to \HFred(S^3_0(\Wh^-(K)))\]
obtained by changing the sign of the clasp in Figure \ref{fig:1-2-hcobordism} is the zero map.  
\end{theorem}

\begin{proof}
In this case the analog of Figure \ref{fig:KirbyY} becomes $J' \subset Y'$, where $Y'$ is $0$-surgery on the knot with a negative clasp, i.e. the figure-eight knot. Then an application of the mapping cone gives
\[ \HF^+(Y') = \cT^+_{(1/2)} \oplus \cT^+_{(-1/2)} \oplus \F_{(-1/2)}. \]
Since $Y'_{-1}(J') = S^1 \times S^2$, we have that $\HF^+(Y'_{-1}(J')) = \cT^+_{(1/2)} \oplus \cT^+_{(-1/2)}$. Hence $J' \subset Y'$ is a genus one fibered knot (fiberedness follows from the fact that the Whitehead link is fibered) on which $-1$-surgery yields a manifold with the same $d$-invariants as $Y'$ and trivial $\HFred$. Then it is straightforward to deduce that $\CFKi(Y', J')$ is the complex in Figure \ref{fig:J'inY'}, tensored with $\F[U, U^{-1}]$.

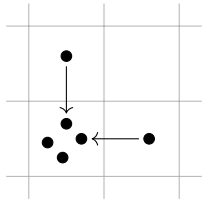
\begin{figure}[ht]
\begin{tikzpicture}[scale=1]
	\draw[step=1, black!30!white, very thin] (-0.3, -0.3) grid (2.3, 2.3);

	\filldraw (0.25, 0.45) circle (2pt) node[] (a) {};
	\filldraw (0.45, 0.25) circle (2pt) node[] (b) {};
	
	\filldraw (0.5, 1.6) circle (2pt) node[] (c) {};
	\filldraw (0.5, 0.7) circle (2pt) node[] (d) {};
	
	\filldraw (1.6, 0.5) circle (2pt) node[] (e) {};
	\filldraw (0.7, 0.5) circle (2pt) node[] (f) {};

	\draw[->] (c) to node[]{} (d);
	\draw[->] (e) to node[]{} (f);

\end{tikzpicture}
\caption{The knot Floer complex of $J' \subset Y'$.}
\label{fig:J'inY'}
\end{figure}

Applying Proposition~\ref{prop:WhiteheadK} below combined with taking mirrors, we can compute the knot Floer complex of the negative Whitehead double of $K$.  We can now run similar arguments as above (that is, applying the K\"unneth formula to find $\CFKm(Y', J' \# K)$, then applying  the mapping cone formula, followed by considering the exact triangle), and the analog of \eqref{eq:exactred} becomes

\begin{equation*}
\begin{tikzpicture}[baseline=(current  bounding  box.center)]
\node(1)at (-3,2){$\F^{2n}$};
\node(2)at (3,2){$\F^{4n}$};
\node(3)at(0,0){$\F^{6n}$.};
\path[->](1)edge node[above]{$F$}(2);
\path[->](2)edge node[right]{$\Phi$}(3);
\path[->](3)edge node[left]{$\Psi$}(1);
\end{tikzpicture}
\end{equation*}

Looking at the ranks, we see that we do not even need to keep track of gradings in order to conclude that the map $F$ is the zero map, by exactness.
\end{proof}

%% file: Whitehead.tex
\section{Whitehead doubles}\label{sec:Whitehead}
In light of Theorem \ref{thm:0-surgerycobordism}, we would like to understand the knot Floer homology of Whitehead doubles.  In order to do this, we first set some notation.  

Let $\HFKhat_\red(K) = \bigoplus_{m,s} \HFKhat_{\red, m}(K,s)$ where
\begin{align*}
	\HFKhat_{\red, m}(K,s) =
	\begin{cases}
		\F_{(0)}^{-1} \oplus \HFKhat_m(K,s)  &s= \tau(K) \\
		\HFKhat_m(K,s)  &s\neq \tau(K).
	\end{cases}
\end{align*}
That is, $\HFKhat_\red(K)$ can be thought of as $\HFKhat(K)$ modulo a generator that survives in the spectral sequence from $\HFKhat(K)$ to $\HFhat(S^3)$.  We need to also understand $H_*(\cF(K,i))$ when $\tau(K)=0$. We assume that $\CFKhat(K)$ is reduced, which means we can choose a $\Z$-filtered basis for the $\Z$-filtered chain complex $\CFKhat(K)$, that is, a basis of the form
\[ x, y_i, z_i \]
for some finite index set $\{i\}_{i\in I}$ with the following gradings and differential:

\begin{center}
\vspace{5pt}
\begin{tabular}{*{10}{@{\hspace{10pt}}c}}
\hline
& && $m$ && $A$ && $\d$  & \\
\hline
& $x$ && $0$ && $0$ && $0$\\ 
& $y_j$ && $m_j$ && $A_j$ && $z_j$\\ 
& $z_j$ && $m_j-1$ && $A_j-d_j$ && $0$\\ 
\hline
\end{tabular}
\vspace{5pt}
\end{center}
With this notation, $x$ is the generator that we quotient by to obtain $\HFKhat_\red$.

\begin{proposition}\label{prop:WhiteheadK}
Let $K$ be a nontrivial knot with $\tau(K)=0$ with knot Floer complex as described above. The knot Floer complex  $\CFKm(\Wh(K))$ is of the form
\[ x \oplus \bigoplus_{j=1}^n (B[m_j-1])^{2d_j}	\]
where $x$ is in Maslov and Alexander grading zero, $B[k_i]$ is the $1 \times 1$ box in Figure \ref{fig:box}, and $\max \{m_j\} = \max \{ m \mid \HFKhat_{\red,m}(K) \neq 0 \}$.
\end{proposition}

\begin{remark}\label{rem:t2n}
Observe that for $K_n = T_{2,n} \# -T_{2,n}$ with $n \geq 3$ and $n$ odd, we have 
\[ \max \{ m \mid \HFKhat_m(K_n) \neq 0 \} = n-1. \]
This follows from the fact that $K_n$ is alternating, which by \cite{OS-alternating} and \cite{Petkova-thin} implies that the knot Floer complex is determined by the Alexander polynomial and signature.
\end{remark}

The proof of this proposition will rely on the following very useful result of Hedden:

\begin{theorem}[{\cite[Theorem 1.2]{Hedden-Whitehead}}]\label{thm:Hedden-Whitehead}
Let $K$ be a slice knot of genus $g$. Then
\[ \HFKhat_*(\Wh(K), s) = 
\begin{cases}
	 \F^{-2g-2}_{(1)} \oplus \bigoplus_{i=-g}^{g} [H_{*-1}(\cF(K, i))]^2 \quad &s=1 \\
	 \F^{-4g-3}_{(0)} \oplus \bigoplus_{i=-g}^{g} [H_{*}(\cF(K, i))]^4 \quad &s=0 \\
	 \F^{-2g-2}_{(-1)} \oplus \bigoplus_{i=-g}^{g} [H_{*+1}(\cF(K, i))]^2 \quad &s=-1.
\end{cases}
\]
\end{theorem}

\begin{proof}[Proof of Proposition \ref{prop:WhiteheadK}]
Using the basis above, we begin by understanding $H_*(\cF(K, i))$.

We first consider the contribution of $x$ to $H_*(\cF(K, i))$. The contribution is zero for $i < 0$ and $\F$ for $0 \leq i \leq g$.
Hence by Theorem \ref{thm:Hedden-Whitehead} the contribution of $x$ to $\HFKhat_*(\Wh(K),s)$ is 
\begin{align*}
	\F^{2g+2}_{(1)} \qquad &s=1 \\
	\F^{4g+4}_{(0)} \qquad &s=0 \\
	\F^{2g+2}_{(-1)} \qquad &s=-1.
\end{align*}
Note that these terms are almost exactly offset by the initial negative-exponent terms in Theorem \ref{thm:Hedden-Whitehead}, with the only remaining term being a copy of $\F$ in Maslov and Alexander grading zero.

We next observe that the contribution of each pair $y_j, z_j$ to $H_*(\cF(K, i))$ is $\F$ for $A_j-d_j \leq i \leq A_j-1$ and zero otherwise.  We now use Theorem \ref{thm:Hedden-Whitehead} to consider the contribution of each pair $y_j, z_j$ to $\HFKhat_*(\Wh(K),s)$, which is
\begin{align*}
	&\F^{2d_j}_{(m_j)} \qquad \quad s=1 \\
	&\F^{4d_j}_{(m_j-1)} \qquad  s=0 \\
	&\F^{2d_j}_{(m_j-2)} \qquad s=-1.
\end{align*}

Taking the direct sum of these results yields
\[ \HFKhat_*(\Wh(K), s) = 
\begin{cases}
	 \bigoplus_{j} \F^{2d_j}_{(m_j)} \quad &s=1 \\
	 \F_{(0)} \oplus \bigoplus_{j} \F^{4d_j}_{(m_j-1)} \quad &s=0 \\
	\bigoplus_{j} \F^{2d_j}_{(m_j-2)} \quad &s=-1 
\end{cases}
\]
Finally, the arguments as in \cite[Section 9.1]{CHH-filtering} yield the desired result. That is, for each $j$, we get $2d_j$ copies of the $1 \times 1$ box $B[m_j-1]$.
\end{proof}

%% file: proofs.tex
\section{Proofs of Theorem \ref{thm:exoticR4}, Corollary \ref{cor:exoticR4}, Theorem \ref{thm:branching}, and Theorem \ref{thm:endsumproduct}.}

\begin{proof}[Proof of Theorem \ref{thm:exoticR4}]

Let $K$ be a nontrivial slice knot, let $D$ be a slice disk for $K$, and let $\caR$ denote the exotic $\bbR^4$ constructed by attaching $CH^+$ to the slice disk complement $(B^4,S^3)-(D,K)$ and deleting the upper boundary. Give $\caR$ its unique spin$^c$ structure $\mathfrak{s}$ and note $c_1(\mathfrak{s}) =0$. We will use the exhaustion in Lemma \ref{lem:sliceR4exhaustion} to show the highest graded summand of the end Floer homology $\HE(\caR,\mathfrak{s})$ is governed by the highest grading of $\HFKhat_\red$, from which Item (\ref{it:exoticR4-1}) and the orientation preserving case of Item (\ref{it:exoticR4-2}) immediately follow. To rule out orientation-reversing diffeomorphisms between various $\caR$, it suffices to show there is no orientation-preserving diffeomorphism between any of the $\caR$ and any slice $\R^4$ made from the simplest all-negative clasp Casson handle. Using Theorem \ref{thm:negativedouble} we show the end Floer homology of the latter vanishes, which will complete the proof.  

By Lemma \ref{lem:sliceR4exhaustion}, $\caR$ admits a compact exhaustion $X_i$, $i\ge 0$, where $X_0 = B^4\setminus \nu(D)$ and $X_i = X_0 \cup T_i^+$, where $T_i^+$ is attached along a 0-framed meridian of $D$. Moreover, each $Y_i := \partial X_i \cong S_0^3(\Wh^i(K))$, and each $W_{ij}: = X_j - \text{int} \,X_i$ is obtained by attaching 1-handles and 2-handles to $Y_i$, where all 2-handles are attached along knots representing primitive elements of $H_1(Y_i\# (S^1 \times S^2))$. Thus, by Gadgil's examples of admissible handle attachments \cite[Lemma 2.3]{gadgil}, and the fact that compositions of admissible cobordisms are admissible \cite[Lemma 2.1]{gadgil}, $\{X_i\}$ is an admissible exhaustion of $\caR$. By Lemma \ref{lem:sliceR4exhaustion}, $X_i$ is diffeomorphic to $B^4\setminus \nu(\Wh^i(D))$. From the Mayer-Vietoris sequence for $B^4 = X_i \,\cup \,$(2-handle) we find $b_2(X_i) = b_3(X_i) = 0$ for all $i$.  Hence $\{X_i\}$ is a grading-admissible exhaustion of $(\caR, \mathfrak{s})$. Since $b_1(Y_i)=1$ for all $i$, Theorem \ref{thm:absgr} gives that
\[\HE(\caR, \mathfrak{s}) = \varinjlim \HFred(Y_i, {\mathfrak{s}}|_{Y_i})\left[ -1/2\right]\] has an absolute $\mathbb{Q}$-grading independent of the choice of exhaustion. 

Direct limits are unchanged under passing to subsequences, and so we may take the first module in the directed system to be $\HFred(Y_1, {\mathfrak{s}}|_{Y_1})\left[ -1/2\right]$. We have $Y_1 \cong S^3_0(\Wh(K))$, and $\Wh(K)$ is nontrivial since $K$ is nontrivial (e.g., by considering bridge number as in \cite{schultensbridge, horst}). Thus, by Proposition \ref{prop:WhiteheadK}, we conclude the knot Floer complex $\CFK^-(\Wh(K))$ is of the form
\[ x \oplus \bigoplus_{i=1}^n B[k_i]	,\]
 with at least one box. Applying the mapping cone formula for 0-surgeries, we find $\HFred(Y_1,\s|_{Y_1}) \ne 0$, and the highest graded summand is $\mathbb{F}^k_{(b-\frac{1}{2})}$, where
\[b := \max \{ m \mid \HFKhat_{\red, m}(K) \neq 0 \}-1 .\]

Combining this with item (\ref{it:0-surgerycobordism1}) of Theorem \ref{thm:0-surgerycobordism} shows that for each $i\ge 1$, the highest graded summand of $\HFred(Y_i, {\mathfrak{s}}|_{Y_i})\left[ -1/2\right]$ is $\mathbb{F}^{k\cdot2^{i-1}}$ in grading $(b - 1)$. By item (\ref{it:0-surgerycobordism2}) of Theorem \ref{thm:0-surgerycobordism}, for each $i\ge 1$ the cobordism map 
\[F_{W_{i,i+1},\mathfrak{s}|_{W_{i,i+1}}} \colon \HFred(S^3_0(\Wh^i(K))) \to \HFred(S^3_0(\Wh^{i+1}(K)))\]
is injective on the highest graded summand of the domain. Since each $W_{i,i+1}$ is obtained by attaching a 1- and 2-handle to $Y_i\times I$, we compute that its grading shift is 0. It follows that $\HE(\caR, \mathfrak{s})$ contains an $\mathbb{F}^\infty$ summand in grading $(b - 1)$, and this is the highest grading of any nontrivial homogeneous element. 

Since the graded group $\HE(\caR,\mathfrak{s})$ is an invariant of the diffeomorphism type of $\caR$, and the standard $\bbR^4$ has vanishing end Floer homology \cite[Proposition 1.5]{gadgil}, we conclude $\caR$ is exotic. This establishes Item (\ref{it:exoticR4-1}). The orientation preserving case of Item (\ref{it:exoticR4-2}) follows immediately by our maximal grading computation. To justify the orientation reversing case of Item (\ref{it:exoticR4-2}) suppose $\caR'$ is constructed from a nontrivial slice disk complement union the simplest Casson handle with all negative clasps. Proceeding as above, we find a grading admissible exhaustion for $\caR'$ where the cobordism maps are as in Theorem \ref{thm:negativedouble}. Thus, $\HE(\caR',\s) =0$ since it is a direct limit where all maps are zero. Therefore $\caR\not\cong \caR'$ for any $\caR$.
\end{proof}

\begin{remark}\label{rem:startlate}
Since direct limits are unchanged by passing to subsequences, we could have taken the first module in the directed system to be $\HFred(Y_m, {\mathfrak{s}}|_{Y_m})\left[ -1/2\right]$ for any $m$. Thus, under the same assumptions on $K$, the conclusion that $\HE(\caR, \mathfrak{s})$ has highest graded summand $\mathbb{F}^\infty_{(b-1)}$ holds if we built $\caR$ by starting with \textit{any} disk for $\Wh^m(K)$ instead of $K$.
\end{remark}

Note that the choice of slice disk does not affect the end Floer homology, even if one uses non-isotopic disks. If one constructs slice $\R^4$'s using the same slice knot, two nondiffeomorphic slice disk complements, and $CH^+$, then the two end Floer homologies coincide.

\begin{proof}[Proof of Corollary \ref{cor:exoticR4}]
Let $K_n = T_{2,n}\#-T_{2,n}$ where $n\ge 3$ is odd, let $\CH^+$ be the Casson handle with a single positive plumbing at each stage, and let $\caR_n$ denote the ribbon $\bbR^4$ constructed by attaching $CH^+$ to the standard disk complement for $K_n$, and deleting the upper boundary. Give each $\caR_n$ its unique spin$^c$-structure $\mathfrak{s_n}$. By Theorem \ref{thm:exoticR4} and Remark \ref{rem:t2n} we find $\HE(\caR_n, \mathfrak{s}_n)$ and $\HE(\caR_m, \mathfrak{s}_m)$ are not isomorphic for odd $n\ne m$, with $n,m \ge 3$.
\end{proof}

\begin{proof}[Proof of Theorem \ref{thm:branching}.]
(\ref{it:branching1}) Let $\CH$ be a Casson handle containing an infinite positive chain and let $\caR$ be as in the statement of the theorem. Let $\caR^+$ be the exotic $\R^4$ made by attaching $\CH^+$ to the same disk complement. Let $\{X_n^+\}$ be the standard exhaustion for $\caR^+$, and notice $\caR$ has a standard exhaustion $\{X_n\}$ similar to that of Lemma \ref{lem:sliceR4exhaustion} where $X_n$ is obtained from the disk complement $B^4\setminus\nu(D)$ by attaching the first $n$ stages of plumbed 2-handles of $\CH$. This provides an admissible exhaustion of $\caR$.  Among the top stage plumbed handles of $X_n$, exactly one comes from the positive chain; by plugging all others with 2-handles attached along 0-framed meridians, we find that $X_n^+ = X_n \cup Z_n$ where $Z_n$ is the cobordism obtained from the 2-handle attachments. Thus, the cobordism maps $F_{W^+_{1,j}}$ of Theorem \ref{thm:0-surgerycobordism} factor as $F_{Z_n} \circ F_{W_{1,j}}$. By our assumption that $K$ is a Whitehead double, Theorem \ref{thm:0-surgerycobordism} shows that the maps $F_{W^+_{1,j}}$ are injective on the highest graded summand of $\HFred(S^3_0(K))$. Thus it follows that $F_{W_{1,j}}$ must be as well, hence the end Floer homology of $\caR$ is nonzero.

(\ref{it:branching2}) Suppose $\CH$ has an infinite positive and an infinite negative chain. By part (\ref{it:branching1}) above, the end Floer homology of $\caR$ is nonzero. Note that the reverse orientation manifold $\overline{\caR}$ is a slice $\R^4$ obtained by attaching a Casson handle with an infinite positive chain (the mirror image of the infinite negative chain of the initial Casson handle) to a nontrivial slice disk complement, where the knot is a Whitehead double of a nontrivial slice knot. This is still a genus 1 slice knot whose knot Floer complex $\CFK^-$ is a sum of a single generator and $1\times 1$ boxes, so by Remark \ref{rem:0-surgerycobordism}, the slice $\R^4$ made from combining this disk complement with $CH^+$ has nonvanishing end Floer homology. Thus, we may repeat the proof of part (\ref{it:branching1}) above to conclude $\overline{\caR}$ has nontrivial end Floer homology. This property is not shared by the slice $\R^4$'s of Theorem \ref{thm:exoticR4}: we saw that when their orientation is reversed, their end Floer vanishes. 
\end{proof}

\begin{remark}
In the proof above we would be able to drop the assumption that $K$ is a Whitehead double, if we knew that the injectivity conclusion of Theorem \ref{thm:0-surgerycobordism} still holds.
\end{remark}

\begin{remark}\label{rem:mixedsigns}
In the above proof, the end Floer homology was seen to not vanish in the case where the initial slice knot had been positively or negatively Whitehead doubled. In general, this allows us to extend the proof of item (\ref{it:exoticR4-1}) of Theorem \ref{thm:exoticR4} to include linear-chain Casson handles with only finitely many double points of one sign. This is because such a slice $\R^4$ is diffeomorphic to one made with $\CH^+$ or $\CH^-$ and a positive or negative Whitehead double of a nontrivial slice knot. The current method does not apply when there are infinitely many positive and negative double points in a linear-chain Casson handle.
\end{remark}

\begin{proof}[Proof of Theorem \ref{thm:endfloerzero}]
We prove a slightly more general statement. Let $K$ and $K'$ be slice knots with chosen slice disks. Let $\caR$ be obtained by combining the disk complement of $K$ with $\CH^+$ and let $\caR'$ be obtained by combining the disk complement of $K'$ with $\CH^-$. Then $\caR\natural\caR'$ has a standard admissible exhaustion $\{X_n\}$ similar to that of Lemma \ref{lem:sliceR4exhaustion}. Thus to compute the end Floer homology with the only torsion spin$^c$ structure $\s$, the relevant cobordism maps are the 1- and 2-handle cobordism maps appearing in Theorems \ref{thm:0-surgerycobordism} and \ref{thm:negativedouble}:
\[F_{W_{i,i+1},\mathfrak{s}|_{W_{i,i+1}}} \colon \HFred(S^3_0(\Wh_+^i(K))\#S^3_0(\Wh_-^i(K'))) \to \HFred(S^3_0(\Wh^{i+1}_+(K))\#S^3_0(\Wh_-^{i+1}(K'))).\]
By the K\"unneth formula and the vanishing property of the negative double cobordism map, proved in Theorem \ref{thm:negativedouble}, it follows that all of these maps are zero, hence the end Floer homology vanishes. Similarly, the relevant maps for the reverse orientation manifold are zero.

These manifolds are clearly homeomorphic to $\R^4$. They are exotic since $\caR$ made from $\CH^+$ is, and there are no inverses with respect to end-sum \cite[Corollary A.3]{gompfinfinite}. To obtain the theorem statement, note that choosing $\caR' = \overline{\caR}$ amounts to choosing $K' = -K$.
\end{proof}

To prove Theorem \ref{thm:endsumproduct} we wish to use the exhaustion constructed in Lemma \ref{lem:YxRexhaustion} to compute the end Floer homology for the nonstandard end of $(Y\times \bbR)\natural \caR_n$. However, this exhaustion is not grading admissible in the sense of Definition \ref{def:gradingadmissible}; rather, it satisfies a relative version of grading admissibility which we define presently. We will show this condition still allows the construction of a well-defined and absolutely graded direct limit invariant; moreover this condition allows the computation of $\HE(X)$ for individual ends of a 4-manifold relative to a 3-manifold ``cut''. 

\begin{definition}\label{def:relativegradingadmissible}
Let $(X, Y_0,\mathfrak{s})$ be a smooth, oriented noncompact 4-manifold with exactly one end, exactly one boundary component $Y_0$, and a spin$^c$ structure $\mathfrak{s}$. We require that $Y_0$ is closed. Suppose $\{(X_i,Y_0)\}$ is a compact exhaustion of $(X,Y_0)$ such that $X_i \subset \mathrm{int}(X_{i+1})$ for all $i$. Let $Y_i := \partial X-  Y_0$ and $W_{ij} := X_j - \mathrm{int}(X_i)$ for $j>i$, with $Y_i$ connected.  We say the exhaustion $\{(X_i,Y_0)\}$ is \textit{grading admissible} if the exhaustion $\{X_i\}$ is admissible in the sense of Gadgil, and also satisfies 
\begin{enumerate}
\item $H_2(X_i,Y_0;\mathbb{Q}) = H_3(X_i,Y_0;\mathbb{Q}) =0$ for all $i > 0$, and
\item $\mathfrak{s}|_{Y_i}$ is torsion for all $i > 0$.
\end{enumerate}
\end{definition}

\begin{theorem}\label{thm:relativegradingadmissible}
Let $(X,Y_0, \mathfrak{s})$ be a smooth 4-manifold with a spin$^c$ structure as in Definition \ref{def:relativegradingadmissible}, and suppose it admits a grading adimssible exhaustion $\{(X_i,Y_0)\}$. Then $(\HFred(Y_i, \mathfrak{s}|_{Y_i}), F_{W_{ij}})$ is a direct system, and the direct limit
\[ \HE(X,Y_0, \s) := \varinjlim \HFred(Y_i, \s|_{Y_i})[-b_1(Y_i)/2] \]
is an absolutely $\mathbb{Q}$-graded $\mathbb{F}[U]$-module that does not depend on the choice of grading admissible exhaustion.
\end{theorem}

\begin{remark}\label{rem:relativeendinvariant}
Since direct limits are unchanged under passing to subsequences, and since any neighborhood of the end of $X$ must contain a subsequence of the $Y_i$, it follows that $\HE(X,Y_0,\mathfrak{s})$ is an asymptotic spin$^c$ diffeomorphism invariant of the end of $X$.  \end{remark}

The relative version of grading admissibility allows us to prove an analogue of Lemma \ref{lem:b1XXYY}, in the same manner. 

\begin{lemma}\label{lem:relativeb1XXYY}
Let $\{(X_i,Y_0)\}$ be a grading admissible exhaustion of $(X,Y_0)$. Then for all $i \ge 0$,
\[b_1(X_i) = b_1(Y_i).\]
\end{lemma}

\begin{proof}
Consider the following portion of the long exact sequence of the triple $(X_i, Y_0\sqcup Y_i, Y_i)$ over $\mathbb{Q}$:
\[H_3(X_i, Y_i) \to H_3(X_i, Y_0\sqcup Y_i) \to H_2(Y_0\sqcup Y_i,Y_i) \to H_2(X_i, Y_i).\]
By grading admissibility, the first and last vector spaces are trivial, hence the middle two vector spaces are isomorphic. We have
\begin{align*}
H^1(X_i) \cong H_3(X_i, Y_0\sqcup Y_i)
\cong H_2(Y_0\sqcup Y_i,Y_i)
\cong H_2(Y_0)
\cong H^1(Y_0),
\end{align*}
over $\mathbb{Q}$.
\end{proof}

\begin{proof}[Proof of Theorem \ref{thm:relativegradingadmissible}]
By Gadgil admissibility and the fact that we are working over $\mathbb{F} = \mathbb{Z}_2$, we obtain the desired direct system and direct limit. We must verify the direct limit has an absolute grading and that it is independent of the choice of exhaustion. 

First we show that the conclusion of Proposition \ref{prop:grshift} holds in our situation: the proof proceeds almost identically, the only differences being that we use Lemma \ref{lem:relativeb1XXYY} to conclude $\chi(W_{ij}) = b_1(Y_i) - b_1(Y_j)$, and that we must show $\sigma(W_{ij}) =0$. It suffices to show the intersection pairing on each $H_2(X_i)$ is zero. Consider the following portion of the long exact sequence of the pair $(X_i, Y_0)$ over $\Q$:
\[H_3(X_i, Y_0) \to H_2(Y_0) \to H_2(X_i) \to H_2(X_i, Y_0).\]
Since the first and last terms are zero, we have $H_2(Y_0,\Q)\cong H_2(X_i, \Q)$ for any $i$, and the isomorphism is induced by inclusion.  Therefore, the intersection form for $H_2(X_i)$ is determined by what it does on surfaces in $Y_0$. Because any surface in $Y_0$ has self-intersection zero in $X_i$ (push it into the collar), we see the intersection form on $H_2(X_i)$ is zero. 

The remainder of the proof that the direct limit admits an absolute grading, and that the direct limit is independent of admissible exhaustion chosen, is identical to the proof of Theorem \ref{thm:absgr}. Note that it is clear that interweaving exhaustions rel $Y_0$ maintains the grading admissible condition.
\end{proof}

\begin{proof}[Proof of Theorem \ref{thm:endsumproduct}]
Let $M$ be a fixed closed, connected 3-manifold. We construct infinitely many smoothings of $M\times \bbR$ by end-summing with the exotic $\bbR^4$'s $\caR_n$ from Corollary \ref{cor:exoticR4}. We will distinguish these smoothings by examining their relative end Floer homologies computed from both ends.

First suppose $M$ is orientable, and consider $(M\times \bbR) \natural \caR_n$ where the end-sum is performed along a standard ray $\{p\}\times [1,\infty)\subset (Y\times \bbR)$. Fix an embedding $M = M\times\{0\}\subset(M\times \bbR) \natural \caR_n$ and let $\mathfrak{S}$ be the set of torsion spin$^c$ structures on $(M\times \bbR) \natural \caR_n$.
 Since any diffeomorphism of $(M\times \bbR) \natural \caR_n$ sends ends to ends, by Remark \ref{rem:relativeendinvariant} it must preserve the set of end floer homologies of the two ends, where we sum over torsion spin$^c$ structures:
\[\mathcal{HE}((M\times \bbR) \natural \caR_n, \mathfrak{S}) := \left\{\bigoplus_{\s \in \mathfrak{S}}\HE(M\times (-\infty,0],M,\s_L), \bigoplus_{\s \in \mathfrak{S}}\HE((M\times [0,\infty)) \natural \caR_n, M,\s_R)\right\}.\] 
(Here we write $\s_L,\s_R$ to denote the appropriate restrictions of $\s$.) The standard end has an exhaustion by product cobordisms between copies of $M$: these are clearly grading admissible in the relative sense, therefore
\[\bigoplus_{\s \in \mathfrak{S}}\HE(M\times (-\infty,0],M,\s_L) \cong \bigoplus_{\s \in \mathfrak{S}}\HFred(M,\s|_{M})[-b_1(M)/2]\]
as absolutely $\Q$-graded $\mathbb{F}[U]$-modules. Note, this has a fixed highest graded nontrivial summand, whose grading does not depend on $n$. 

Lemma \ref{lem:YxRexhaustion} yields a compact exhaustion $\{(X_i,M)\}$ for $((M\times [0,\infty)) \natural \caR_n, M)$, for which the corresponding system of cobordisms
\[W_{i,i+1} \colon M \# S^3_0(\Wh^i(K_n)) \to M \# S^3_0(\Wh^{i+1}(K_n))\] 
is grading admissible. By the K\"unneth formula for Heegaard Floer homology and item (\ref{it:0-surgerycobordism1}) of Theorem \ref{thm:0-surgerycobordism}, it follows that if $n$ is sufficiently large relatively to the gradings of $HF^+(M)$ in torsion spin$^c$ structures, then for all $i\ge 1$,
\begin{enumerate}
	\item the highest graded nontrivial summand of $\HFred(M \# S^3_0(\Wh^i(K_n))\#S^1 \times S^2)$, restricted to torsion spin$^c$ structures, has grading $(n + f(M))$, and
	\item the highest graded nontrivial summand of $\HFred(M \# Y_{-1}(J\#\Wh^i(K_n)))$, restricted to torsion spin$^c$ structures, has grading $(n + f(M)-\frac{1}{2})$,
\end{enumerate}
where $f(M)$ is a number that only depends on $\HF^+(M)$.
Using the surgery exact triangle as in the proof of Lemma \ref{lem:gradinginjection}, we find that the cobordism maps $W_{ij}$ are injective on the highest graded summand of $\HFred(M \# S^3_0(\Wh^i(K_n)))$, restricted to torsion spin$^c$ structures. Thus, by keeping track of the grading shift, we find that the highest graded nontrivial summand of $\bigoplus_{\s \in \mathfrak{S}}\HE((M\times [0,\infty)) \natural \caR_n, M,\s_R)$ lives in grading $(n + f(M) -1- b_1(M)/2)$. Thus, by considering all sufficiently large $n$, we obtain infinitely many distinct sets $\mathcal{HE}((M\times \bbR) \natural \caR_n, \mathfrak{S})$, and hence infinitely many diffeomorphism types $(M\times\R)\natural \caR_n$.

Next suppose $M$ is non-orientable: let $M'$ be the orientation double cover and consider $(M'\times \bbR)\natural_2 \caR_n$, where the end sum is performed along the lift of a fiber ray. This orientable manifold is a double cover of $(M\times \bbR) \natural \caR_n$, where the end sum is performed along a fiber ray. Note that a diffeomorphism $(M\times \bbR) \natural \caR_n \to (M\times \bbR) \natural \caR_m$ induces a diffeomorphism of the covers $(M'\times \bbR)\natural_2 \caR_n \to (M'\times \bbR)\natural_2 \caR_m$. Let $\mathfrak{S}$ be the set of torsion spin$^c$ structures on $(M'\times \bbR)\natural_2 \caR_n$. By a minor modification of Lemma \ref{lem:YxRexhaustion}, we obtain a grading admissible exhaustion of the nonstandard end, as a stack of cobordisms:
\[W_{i,i+1}: M \#_2 \,S^3_0(\Wh^i(K_n)) \to M \#_2\,S^3_0(\Wh^{i+1}(K_n)).\]
By similar reasoning as in the orientable case, by varying across sufficiently large $n$, we obtain distinct sets $\mathcal{HE}((M'\times \bbR)\natural_2 \caR_n ,\mathfrak{S}).$
\end{proof}

\begin{remark}
We used $\caR_n$ to obtain new smooth structures on $M\times \R$ for convenience; any ribbon $\R^4$'s distinguished in Theorem \ref{thm:exoticR4} with sufficiently large gradings relative to those of $\HF^+(M)$ will give distinct smoothings of $M\times \R$. Note that our approach only yields countably many smoothings.
\end{remark}

\begin{remark}\label{rmk:endsumRn}
The proof of Theorem \ref{thm:endsumproduct} allows us to distinguish between various end-sums of the form $\caR_{n_1}\natural\hdots\natural\caR_{n_k}$: using the K\"unneth formula as above we can determine the maximal grading for nontrivial elements of the respective end Floer homologies and obtain different values. For example, with notation as in Corollary \ref{cor:exoticR4}, the highest nontrivial gradings in the end Floer homologies of $\caR_a\natural \caR_b$ and $\caR_a\natural \caR_c$ differ by $|b-c|$.
\end{remark}